\documentclass[]{article}

\usepackage[utf8]{inputenc}
\usepackage[margin=1.25in]{geometry}

\usepackage{amsthm, amsfonts, amsmath, makecell, tikz, 
	graphicx, colortbl}
\usepackage{multirow, booktabs, bigstrut}

\usepackage{mathtools}

\usepackage{graphicx}
\usepackage{float}
\usepackage{color}
\usepackage{tikz,pgfplots}
\usepackage{tikz-cd}
\usepackage{verbatim}
\usepackage{enumerate}
\usepackage{physics}
\usepackage{braket}

\usepackage{bm}

\usepackage{enumerate, amssymb}

\usetikzlibrary{graphs,graphs.standard,arrows.meta,calc,intersections}

\usepackage[]{algorithm, algorithmic}

\usepackage{tabularx}

\theoremstyle{plain}
\newtheorem{theorem}{Theorem}[section]
\newtheorem{lemma}[theorem]{Lemma}

\newtheorem{proposition}[theorem]{Proposition}

\theoremstyle{definition}
\newtheorem{definition}[theorem]{Definition}
\newtheorem{example}[theorem]{Example}

\DeclareMathOperator{\Int}{Int}

\DeclareMathOperator{\IntR}{Int{}^\text{R}}

\DeclareMathOperator{\minval}{minval}

\renewcommand{\epsilon}{\varepsilon}

\newcommand{\R}{{\mathbb R}}
\newcommand{\N}{{\mathbb N}}

\newcommand{\Q}{{\mathbb Q}}
\newcommand{\F}{{\mathbb F}}
\newcommand{\Z}{{\mathbb Z}}

\newcommand{\m}{\mathfrak{m}}
\newcommand{\M}{\mathfrak{M}}

\definecolor{gray}{rgb}{.5,.5,.5}

\definecolor{black}{rgb}{0,0,0}

\definecolor{blue}{rgb}{0,0,1}

\definecolor{red}{rgb}{1,0,0}

\definecolor{green}{rgb}{0,1,0}

\definecolor{gold}{rgb}{.5,.5,.2}

\definecolor{yellow}{rgb}{1,1,.4}

\definecolor{purple}{rgb}{.5,0,.5}

\definecolor{darkgreen}{rgb}{0,.5,0}

\definecolor{orange}{rgb}{1,.55,0}

\definecolor{white}{rgb}{1,1,1}

\usepackage[colorinlistoftodos]{todonotes}

\let\originalleft\left
\let\originalright\right
\renewcommand{\left}{\mathopen{}\mathclose\bgroup\originalleft}
\renewcommand{\right}{\aftergroup\egroup\originalright}

\let\originaltodo\todo
\renewcommand{\todo}[1]{\originaltodo[inline]{#1}}

\usepackage{hyperref}
\usepackage{xcolor}
\hypersetup{
	colorlinks,
	linkcolor={red!50!black},
	citecolor={blue!50!black},
	urlcolor={blue!80!black}
}

\usepackage{subfigure}

\title{Factorization in rings of integer-valued rational functions}

\author{Baian Liu}

\begin{document}
	
	\maketitle

\begin{abstract}
	For a domain $D$, the ring $\Int(D)$ of integer-valued polynomials over $D$ is atomic if $D$ satisfies the ascending chain condition on principal ideals. However, even for a discrete valuation domain $V$, the ring $\IntR(V)$ of integer-valued rational functions over $V$ is antimatter. We introduce a family of atomic rings of integer-valued rational functions and study various factorization properties in these rings. 
\end{abstract}

\section{Introduction}
\indent\indent Although the ring of integer-valued rational functions over a domain has a definition that is closely related to the more much well-studied ring of integer-valued polynomials, the two rings can behave very differently depending on the base ring. Some of these differences are revealed through the lens of factorization. 

Given a domain $D$ with field of fractions $K$ and $E$ some subset of $K$, we first define 	\[
\IntR(D) \coloneqq \{ \varphi \in K(x) \mid  \forall d \in D, \varphi(d) \in D \} \quad \text{and} \quad	\IntR(E,D) \coloneqq \{ \varphi \in K(x) \mid \forall a \in E, \varphi(a) \in D\}
\]
the \textbf{ring of integer-valued rational functions over $D$} and the \textbf{ring of integer-valued rational functions on $E$ over $D$}, respectively.	Note that $\IntR(D,D) = \IntR(D)$. Compare these definitions to 
\[
\Int(D) \coloneqq \{ f \in K[x]\mid \forall d \in D, f(d) \in D \} \quad \text{and} \quad	\Int(E,D) \coloneqq \{ f \in K[x] \mid \forall a \in E, f(a) \in D\}
\]
the \textbf{ring of integer-valued polynomials over $D$} and \textbf{ring of integer-valued polynomials on $E$ over $D$}, respectively. 

The ring of integer-valued rational functions has ideals that can be defined using the ideals of the base ring. The ideals of $\IntR(E,D)$ relevant to this work are
\[
	\M_{ \m, a } \coloneqq \{\varphi \in \IntR(E,D) \mid \varphi(a) \in \m \} \quad \text{and} \quad \IntR(E,\m) \coloneqq \{\varphi \in \IntR(E,D) \mid \forall a \in E, \varphi(a) \in \m\}
\]
where $\m$ is some maximal ideal of $D$ and $a$ is some element of $E$. Moreover, $\M_{ \m, a }$ is a maximal ideal of $\IntR(E,D)$. 

Whenever we refer to a ring, we are indicating a commutative ring with identity. Likewise, whenever we refer to a monoid, we are indicating a commutative monoid. For a ring $R$, we denote by $R^\bullet$ the multiplicative monoid of nonzero elements of $R$ and $R^\times$ the multiplicative monoid of units of $R$. We also use $\N$ to denote the set of natural numbers including $0$. For a totally ordered abelian group $\Gamma$, we view it as being embedded in its divisible closure $\Q\Gamma \coloneqq \Gamma \otimes_{\Z} \Q$. We also define $\Gamma_{\geq 0} \coloneqq \{\gamma \in \Gamma \mid \gamma \geq 0\}$.

To study factorization in a ring is to study how elements of the ring can be written as a product of irreducible elements, or \textbf{atoms}. A ring in which every nonzero, nonunit element can written as a product of finitely many atoms is \textbf{atomic}. On the other end of the spectrum, if a ring has no atoms, we call it \textbf{antimatter}.

Factorization in rings of integer-valued polynomials has been studied extensively. This can take the form of factorization over general rings of integer-valued polynomials, such as in \cite{RingsBetween,GottiLi}. Factorization over the classical ring of integer-valued polynomials $\Int(\Z)$ has also been studied, with Frisch showing that given any finite multiset of integers greater than or equal to 2, there exists some polynomial in $\Int(\Z)$ whose factorization lengths is exactly the given multiset \cite{Frisch}. This result has been extended to $\Int(D)$, where $D$ is a Dedekind domain with finite residue fields and infinitely many maximal ideals \cite{LengthsDedekind}.

There is seemingly no content in investigating factorization in rings of integer-valued rational functions since even for any discrete valuation domain $V$, the ring $\IntR(V)$ is antimatter \cite[Proposition X.3.3]{Cahen}. However, we introduce a family of domains $D$ with field of fractions $K$ such that $\IntR(K,D)$ is atomic. This allows us to study factorization invariants over this family of rings of integer-valued rational functions.

Let $R$ be a ring and $r \in R$ be a nonzero, nonunit element. Since only the multiplicative structure of the ring is considered in these definitions, these definitions also apply to a commutative monoid. We say two elements $a, b \in R$ are \textbf{associates} if there exists a unit $u \in R$ such that $a = ub$, and we write $a \sim b$. We denote by $\mathcal{A}(R) \coloneqq \{a \in R \mid \text{$a$ is an atom}\}/\sim$. The \textbf{set of factorizations} of $r$ in $R$ is given by 
\[
\mathsf{Z}(r) = \left\{ \left(e_{[a]}\right)_{[a] \in \mathcal{A}(R)} \in \bigoplus_{[a] \in \mathcal{A}(R)} \N \,\middle\vert\, \prod_{[a] \in \mathcal{A}(R)} a^{e_{[a]}} \sim r  \right\}.
\]  
Note that the product is a finite product since for each $\left(e_{[a]}\right)$, there are only finitely many elements $[a] \in \mathcal{A}(R)$ such that $e_{[a]}$ is nonzero. The product is also well-defined up to association. For each factorization $\left(e_{[a]}\right) \in \mathsf{Z}(r)$, we define $\abs{\left(e_{[a]}\right)} \coloneqq \sum\limits_{[a] \in \mathcal{A}(R)} e_{[a]}$ to be the \textbf{length} of the factorization. The \textbf{set of lengths} of $r$ is then defined to be $\mathcal{L}(r) := \{\abs{z} \mid z \in \mathsf{Z}(r) \}$. We also set $L(r) \coloneqq \sup \mathcal{L}(r)$ and $\ell(r) \coloneqq \inf \mathcal{L}(r)$ to be the \textbf{longest factorization length} and the \textbf{shortest factorization length} of $r$ in $R$. 
We say that $R$ is of \textbf{bounded factorization} if $R$ is atomic and $\abs{\mathcal{L}(s)} < \infty$ for each nonzero, nonunit element $s \in R$. Let $z = \left(e_{[a]}\right)$ and $z' = \left(e_{[a]}'\right)$ be two factorizations in $\mathsf{Z}(r)$. We define $\gcd(z, z') = (\min\{e_{[a]}, e_{[a]}'\})$. The \textbf{distance} between the two factorizations of $r$ in $R$ is defined to be $\max\{\abs{z-\gcd(z,z')}, \abs{z' - \gcd(z,z')} \}$. For $N \in \N$, an \textbf{$N$-chain} between $z$ and $z'$ is a finite sequence of factorizations $z = z_0, \dots, z_n = z'$ such that each $z_i \in \mathsf{Z}(r)$ and $d(z_i, z_{i+1}) \leq N$ for all $i < n$. Then we define the \textbf{catenary degree} of $r$ in $R$, denoted as $c(r)$, to be the minimal $N$ such that there exists an $N$-chain between any two elements of $\mathsf{Z}(r)$. For all of the factorization invariants, a subscript may be used to emphasize the ring. For example, we can use $\mathsf{Z}_R(r)$ to denote the set of factorizations of $r$ in $R$.

Let $D$ be a domain with field of fractions $K$. Suppose that $A$ is a domain such that $D \subseteq A \subseteq K[x]$ and $A \cap K = D$. For example, $A = \Int(E,D)$ for any subset $E$ of $K$ satisfies these conditions. In this case, for any $d \in D$ and any atoms $a_1, \dots, a_n$ of $A$ such that $d = a_1 \cdots a_n$, we have that $a_1, \dots, a_n$ are all atoms of $D$. This is because we may view each $d \in D$ as a degree zero polynomial in $A$ and to write a degree zero polynomial as a product of elements in $A$, we must write the degree zero polynomial as a product of degree zero polynomials, which are elements of $D$. In other words, under this setup, factoring degree zero polynomials in $A$ is exactly the same as factoring them in $D$. However, if we have $D \subseteq A$ but $A$ is not contained in $K[x]$, then we will see that it is possible that an element $d \in D$ can be written as the product of elements of $A$, not all of which are in $D$. For this reason, for our family of domains $D$ such that $\IntR(K,D)$ is atomic, we will mostly focus on factorization of elements $d \in D$ viewed as elements of $\IntR(K,D)$. 

In Section \ref{Sect:Valuation}, we discuss factorization in $\IntR(E,V)$ for $V$ a valuation domain and $E$ a subset of the field of fractions. For Section \ref{Sect:Atomicity}, we introduce some computational tools and use them to prove that for a certain family of domains $D$ with field of fractions $K$, the ring $\IntR(K,D)$ is atomic. As for Section \ref{Sect:Lengths}, we show that the set of factorization lengths of an element $d \in D$ changes if we view $d$ as an element of these atomic rings $\IntR(K,D)$. Lastly, in Section \ref{Sect:1toInfinity}, we focus on factorization of atomic rings of the form $\IntR(K,D)$, where $D$ is a localization of a monoid domain whose underlying monoid has a gap when embedded into its difference group which is a additive subgroup of $\R$ with no smallest strictly positive element. 

\section{Factorization in rings of integer-valued rational functions over valuation domains}\label{Sect:Valuation}

\indent\indent For a domain $D$, if $D$ satisfies the ascending chain condition on principal ideals, then the ring $\Int(D)$ of integer-valued polynomials over $D$ is atomic \cite[Corollary 7.6]{RingsBetween}. However, even for a discrete valuation domain $V$, the ring $\IntR(V)$ is not atomic. In fact, $\IntR(V)$ is an antimatter domain \cite[Proposition X.3.3]{Cahen}. The result in \cite[Proposition X.3.3]{Cahen} states that for any $E$ that is a subset of the field of fractions of a discrete valuation domain $V$, the ring $\IntR(E,V)$ is an antimatter domain.  We will show that this is slightly incorrect in Lemma \ref{Lem:AtomsInIntREV}. We will show that in this case, $\IntR(E,V)$ is an antimatter domain if and only if $E$ does not have an isolated point with respect to the topology induced by the valuation. 

Another reason to consider integer-valued rational functions over valuation domains is that we have tools such as continuity to construct integer-valued rational functions with control in terms of the values the functions attain. 

\begin{proposition}\label{Prop:RationalFunctionsContinuousInValuationTopology}\cite[Proposition 2.1]{Liu}
	Let $D$ be a domain with field of fractions $K$ and $E$ a subset of $K$. Let $v$ be a valuation on $K$ with value group $\Gamma$ such that the associated valuation ring $V$ contains $D$. Then each element of $\IntR(E,D)$ is a continuous function from $E$ to $D$ with respect to the topology induced by the valuation. 
\end{proposition}

The result of \cite[Proposition X.3.3]{Cahen} tells us when $\IntR(E,V)$ is an antimatter domain when $V$ is a discrete valuation domain and $E$ is a subset of the field of fractions of $V$ where $E$ does not $E$ does not have an isolated point with respect to the topology induced by the valuation. This statement is still true if $V$ is a valuation domain with principal maximal ideal. 

\begin{lemma}\label{Lem:StoneWeierstrass}
	Let $V$ be a valuation domain with principal maximal ideal. Suppose that $E$ is some nonempty subset of the field of fractions of $V$. Let $\varphi \in \IntR(E,V)$ be a nonzero, nonunit element. If there exists an $a \in E$ such that $\varphi(a)$ is not a unit and $a$ is not an isolated point of $E$ with respect to the topology induced by the valuation, then $\varphi$ is strictly divisible by a nonunit element of $\IntR(E,V)$. 
\end{lemma}

\begin{proof}
	Let $\varphi \in \IntR(E,V)$ and suppose that there exists an $a \in E$ such that $\varphi(a)$ is not a unit and $a$ is not an isolated point of $E$. By applying the isomorphism $\IntR(E,V) \to \IntR(E-a,V)$ defined by $\rho(x) \mapsto \rho(x + a)$, we can assume without loss of generality that $a = 0$, $0 \in E$, and $0$ is not an isolated point of $E$. 
	
	Let $v$ be the valuation associated with $V$ and $\m$ be the maximal ideal of $V$, generated by some $t \in V$. Since $\varphi(0) \in \m$, there exists some $\delta \in \Gamma$ such that $v(\varphi(d)) \geq v(t)$ for all $d \in E$ with the property that $v(d) > \delta$ due to Proposition \ref{Prop:RationalFunctionsContinuousInValuationTopology}. Because $0$ is not an isolated point, we know that there exists some $b \in E$ such that $v(b) > \delta$. Construct
	\[\psi(x) = \frac{x^3+b^3t^2}{x^3+b^3t}.\]
	Note that $\frac{v(b^3t)}{3} = v(b) + \frac{v(t)}{3} \notin \Gamma$ and similarly $\frac{v(b^3t^2)}{3} \notin \Gamma$. Also notice that there is no element of $\Gamma$ between $\frac{v(b^3t)}{3}$ and $\frac{v(b^3t^2)}{3}$ since there is no element of $\Gamma$ between $\frac{v(t)}{3}$ and $\frac{v(t^2)}{3}$. Then for each $d \in E$, we calculate
	\[
	v(\psi(d)) = \begin{cases}
		0, &\text{if $v(d) < \frac{v(b^3t)}{3}$},\\
		v(b^3t^2)-v(b^3t) = v(t) > 0, &\text{if $v(d) > \frac{v(b^3t^2)}{3}$}.
	\end{cases}
	\]
	Using the fact that $\psi(0) \in \m$, we have that $\psi \in \IntR(E,V)$ and $\psi$ is not a unit of $\IntR(E,V)$. Additionally, for $d \in E$ such that $v(d) < \delta$, we have that $v(\psi(d)) = 0$ since $\delta < v(b) < \frac{v(b^3t)}{3}$. We also know that for $d \in E$ with $v(d) \geq \delta$ that $v(\psi(d)) \leq v(t)$. This implies that $\frac{\varphi}{\psi} \in \IntR(E,V)$ since $v\left(\frac{\varphi}{\psi}(d)\right) \geq 0$ regardless whether $v(d) < \delta$ or $v(d) \geq \delta$. Lastly, we check that $v\left(\frac{\varphi}{\psi}(b)\right) = v(\varphi(b)) > 0$. Thus, $\frac{\varphi}{\psi}$ is not a unit of $\IntR(E,V)$ either, which means that $\varphi$ is strictly divisible by $\psi$. 
\end{proof}

What follows is a technical lemma that constructs integer-valued rational functions over a valuation domain with even more control on the values attained. 

\begin{lemma}\label{Lem:ZigZag}
	Let $V$ be a valuation domain with a maximal ideal that is not principal, and suppose that the residue field of $V$ is algebraically closed or the value group is not divisible. Denote by $K$ the field of fractions of $V$ and $\Gamma$ the associated value group. For any $\alpha, \alpha' \in \Gamma_{\geq 0}$ with $\alpha < \alpha'$ and $\varepsilon \in \Q\Gamma$ with $\varepsilon > 0$, there exist $\varphi_{\alpha, \alpha',\varepsilon} \in \IntR(K,V)$; $\beta, \beta' \in \Q\Gamma$; and $\alpha'' \in \Gamma$ with $\alpha' - \varepsilon < \alpha'' \leq \alpha'$ so that for all $d \in K$, we have
	\[
	v(\varphi_{\alpha, \alpha', \varepsilon}(d)) = \begin{cases}
		\alpha, & \text{if } v(d) < \beta, \\
		\alpha'', & \text{if } v(d) > \beta',
	\end{cases}
	\]
	and $\alpha \leq  v(\varphi_{\alpha, \alpha', \varepsilon}(d)) \leq \alpha''$ if $\beta \leq v(d) \leq \beta'$. If the residue field of $V$ is not algebraically closed and $\Gamma$ is divisible, we can take $\alpha'' = \alpha'$. 
\end{lemma}

\begin{proof}
	Fix $\alpha, \alpha', \varepsilon \in \Gamma_{\geq 0}$. We split into two cases. The first case is concerned with when the residue field of $V$ is not algebraically closed and $\Gamma$ is divisible. The second case is concerned with when $\Gamma$ is not divisible. In both cases, we construct $\varphi_{\alpha, \alpha',\varepsilon}$ directly, but in the first case, the value of $\varepsilon$ is not used in the construction of $\varphi_{\alpha, \alpha',\varepsilon}$. 
	
	\begin{enumerate}
		\item We assume that the residue field of $V$ is not algebraically closed and $\Gamma$ is divisible. Since the residue field of $V$ is not algebraically closed, there exists a monic, nonconstant polynomial with coefficients in the residue field that has no roots. This polynomial lifts to a monic nonconstant polynomial $f \in V[x]$ such that $f(V) \subseteq V^\times$. Let $n \coloneqq \deg(f)$. Take $b \in V$ such that $v(b) = \alpha$ and $c \in V$ such that $v(c) = \frac{\alpha'-\alpha}{n}$. Then we construct
		\[
		\varphi_{\alpha, \alpha', \varepsilon}(x) = \frac{bc^nf\left(\frac{x}{c}\right)}{f(x)}. 
		\]
		Since $v\left(a_1^n f\left(\frac{a_2}{a_1}\right)\right) = \min\{nv(a_1), nv(a_2)\}$ for each $a_1, a_2 \in K$ with $a_1 \neq 0$ \cite[Corollary 2.3]{PruferNonDRings}, we calculate that for each $d \in K$, we have
		\[
		v(\varphi_{\alpha, \alpha', \varepsilon}(d)) = \begin{cases}
			\alpha, & \text{if $v(d) < 0$},\\
			\alpha + nv(d), & \text{if $0 \leq v(d)\leq \frac{\alpha'-\alpha}{n}$},\\
			\alpha + nv(c) = \alpha', & \text{if $v(d) > \frac{\alpha'-\alpha}{n}$}.
		\end{cases}
		\]
		We can confirm that for $d \in K$ such that $0 \leq v(d) \leq \frac{\alpha' - \alpha}{n}$, we have that $0 \leq nv(d) \leq \alpha' - \alpha$ and therefore, $\alpha \leq \alpha + nv(d) \leq \alpha'$. 
		
		This gives the desired function by taking $\beta = 0$, $\beta' = \frac{\alpha'-\alpha}{n}$, and $\alpha'' = \alpha'$. 
		
		\item Now we assume that $\Gamma$ is not divisible. Then there exist some prime $p$ and $\mu \in \Gamma$ with $\mu > 0$ such that $\frac{\mu}{p} \notin \Gamma$. Since the maximal ideal of $V$ is not principal, we have that there exists some $\alpha'' \in \Gamma$ such that $\max\{\alpha' - \epsilon, \alpha\} < \alpha'' \leq \alpha'$ and $\frac{\alpha''-\alpha}{p} \in \Gamma$. Note that $\frac{\mu + \alpha''-\alpha}{p} \notin \Gamma$. We now consider
		\[
		\varphi_{\alpha,\alpha',\varepsilon}(x) = \frac{b(x^p+cc')}{x^p+c'},
		\]
		where $b \in V$ is some element such that $v(b) = \alpha$, $c \in V$ is some element such that $v(c) = \alpha'' - \alpha$, and $c' \in V$ is such that $v(c') = \mu$. For each $d \in K$, we have
		\[
		v(\varphi_{\alpha, \alpha',\varepsilon}(d)) = \begin{cases}
			\alpha, & \text{if $v(d) < \frac{\mu}{p}$},\\
			\alpha + pv(d) - \mu, & \text{if $\frac{\mu}{p} < v(d) < \frac{\mu+\alpha''-\alpha}{p}$ },\\
			\alpha + \alpha'' - \alpha + \mu - \mu = \alpha'',& \text{if $v(d) > \frac{\mu+\alpha''-\alpha}{p}$}.
		\end{cases}
		\]
		Notice that $\frac{\mu}{p} < v(d) < \frac{\mu+\alpha''-\alpha}{p}$ implies that $\alpha < \alpha + pv(d) - \mu < \alpha''$. Thus, we have the desired function taking $\beta = \frac{\mu}{p}$ and $\beta' = \frac{\mu+\alpha''-\alpha}{p}$. 
	\end{enumerate}
\end{proof}

Now we use the previous two lemmas to prove that for a large class of valuation domains $V$, rings of the form $\IntR(E,V)$ is an antimatter domain. 

\begin{proposition}\label{Prop:Antimatter}
	Let $V$ be a valuation domain and let $E$ be a nonempty subset of the field of fractions of $V$. If the residue field of $V$ is not algebraically closed or the value group is not divisible, then $\IntR(E,V)$ is an antimatter domain, unless $E$ has an isolated point with respect to the topology induced by the valuation and the maximal ideal of $V$ is principal. 
\end{proposition}

\begin{proof}
	Let $\m$ denote the maximal ideal of $V$ and let $\Gamma$ denote the value group.
	
	First we do not make any assumptions about the valuation domain $V$ and the subset $E$. Let $\varphi \in \IntR(E,V)$ be a nonzero, nonunit element. Then there exists some $a \in E$ such that $\varphi(a) \in \m \setminus \{0\}$. By setting $\psi(x) := \varphi(x+a) \in \IntR(E-a,V)$, we see that $\psi(0) \in \m \setminus \{0\}$, so we can assume without loss of generality that $\varphi(0) \in \m \setminus \{0\}$ and $0 \in E$. 
	
	Since $\varphi(0) \in \m\setminus \{0\}$, there exists $\delta \in \Gamma$ with $\delta \geq 0$ such that for all $d \in V$ with $v(d) \geq \delta$, we have that $v(\varphi(d)) = v(\varphi(0)) > 0$, due to Proposition \ref{Prop:RationalFunctionsContinuousInValuationTopology}. 
	
	Let $r \in V$ such that $v(r) = \delta$.  Our goal is to construct a nonunit $\psi \in \IntR(E, V)$ such that $\frac{\varphi}{\psi} \in \IntR(E,V)$ and $\frac{\varphi}{\psi}$ is not a unit. This will show that $\IntR(E,V)$ is antimatter. 
	
	We split the proof into three cases. Case 1 is when $V/\m$ is not algebraically closed and $\Gamma$ is divisible. Case 2 is when $\Gamma$ is not divisible and $\m$ is not principal. Case 3 is when $\m$ is principal and $E$ has no isolated points. 
	\begin{enumerate}
		\item We assume that $V/\m$ is not algebraically closed and $\Gamma$ is divisible. We take $\psi(x) = \varphi_{\alpha,\alpha', \varepsilon}\left(\frac{x}{r}\right)$ from Lemma \ref{Lem:ZigZag} using $\alpha = 0$, $\alpha' =  \frac{v(\varphi(0))}{2}$, and any value of $\varepsilon \in \Gamma$ such that $\varepsilon > 0$.	

		We see that $\psi \in \IntR(E,V)$. Since $\psi(0) \in \m$, the rational function $\psi$ is not a unit of $\IntR(E,V)$. Furthermore, if we let $a \in E$, then $v(\psi(a)) = 0$ if $v(a) < \delta$ and $v(\psi(a)) \leq \frac{v(\varphi(0))}{2} < v(\varphi(0)) = v(\varphi(a))$ if $v(a) \geq \delta$. Therefore, $v\left(\frac{\varphi}{\psi}(a)\right) \geq 0$ in both cases, so $\frac{\varphi}{\psi} \in \IntR(V)$. Also, $\frac{\varphi}{\psi}$ is not a unit because 
		\[v\left(\frac{\varphi}{\psi}(0)\right) = v(\varphi(0)) - \frac{v(\varphi(0))}{2}=\frac{v(\varphi(0))}{2} >  0. \]
		
		\item Suppose that $\Gamma$ is not divisible and $\m$ is not principal. 

		Take $\psi(x) = \varphi_{\alpha, \alpha', \varepsilon}\left(\frac{x}{r}\right)$ as in Lemma \ref{Lem:ZigZag} using $\alpha = 0$, $\alpha' =  \frac{v(\varphi(0))}{2}$, and $\varepsilon \in \Gamma$ such that $0 <\varepsilon < \frac{v(\varphi(0))}{2}$.

		From this, we see that $\psi \in \IntR(E,V)$. Since $\psi(0) \in \m$, we also know that $\psi$ is not a unit of $\IntR(E,V)$. Moreover, we have $v\left(\frac{\varphi}{\psi}(a)\right) \geq 0$ for each $a \in E$, since $v(\psi(a)) = 0$ if $v(a) < \delta$ and $v(\psi(a)) < \frac{v(\varphi(0))}{2} < v(\varphi(0)) = v(\varphi(a))$ if $v(a) \geq \delta$. Thus, $\frac{\varphi}{\psi} \in \IntR(E,V)$. Furthermore, $v\left(\frac{\varphi}{\psi}(0)\right) > v(\varphi(0)) - v(\varphi(0)) > 0$, so $\frac{\varphi}{\psi}$ is not a unit of $\IntR(E,V)$.

		\item Now we assume that $\m$ is principal and $E$ has no isolated points with respect to the topology induced by the valuation. We know that $\varphi$ is strictly divisible by another nonunit element of $\IntR(E,V)$ by Lemma \ref{Lem:StoneWeierstrass}.
		
	\end{enumerate}
\end{proof}

If $V$ is a valuation domain with principal maximal ideal and $E$ is some subset of the field of fractions with an isolated point, then the factorization in $\IntR(E,V)$ is more interesting. 

\begin{lemma}\label{Lem:AtomsInIntREV}
	Let $V$ be a valuation domain with field of fractions $K$ and associated valuation $v$. Suppose that the maximal ideal $\m$ of $V$ is principal, generated by some $t \in V$. Take $E$ to be some subset of $K$. Then
	\[
	\mathcal{A}(\IntR(E,V)) = \{[\psi_s] \mid s \in S \},
	\] 
	where $S$ is the set of isolated points of $E$ with respect to the topology induced by the valuation and for each $s \in S$, the rational function $\psi_s$ is given by
	\[
	\psi_s(x) = \frac{(x-s)^3+c_s^3t^2}{(x-s)^3+c_s^3t},
	\]
	with $c_s \in K$ such that there does not exist any $b \in E$ with $v(b-s) > v(c_s - s)$. 
\end{lemma}

\begin{proof}
	Let $s \in S$, then for each $d \in E$, we calculate that
	\[
	v(\psi_s(d)) = \begin{cases}
		0,&\text{if $d \neq s$} \\
		v(t),&\text{if $d = s$}.
	\end{cases}
	\]
	We can deduce that $\psi_s$ is an atom. If $\psi_s = \rho_1\rho_2$ for some $\rho_1, \rho_2 \in \IntR(E,V)$, then $v(t) = v(\psi_s(s)) = v(\rho_1(s)) + v(\rho_2(s))$ implies that $\rho_1(s)$ or $\rho_2(s)$ is a unit of $V$. This implies that $\rho_1$ or $\rho_2$ is a unit of $\IntR(E,V)$, since $\rho_1$ and $\rho_2$ are unit-valued on $E \setminus \{s\}$.

	Now suppose that $\varphi \in \IntR(E,V)$ is an atom. Since $\varphi$ is not a unit, there exists $a \in E$ such that $\varphi(a) \in \m$. If $a \notin S$, then Lemma \ref{Lem:StoneWeierstrass} implies that $\varphi$ is not an atom. Thus, there exists $s \in S$ such that $\varphi(s) \in \m$. Now note that $\psi_s$ divides $\varphi$. Since $\varphi$ is an atom, we have that $\varphi \sim \psi_s$. 
\end{proof}

Therefore, in the case of a valuation domain $V$ with principal ideal, it is possible to obtain a ring of the form $\IntR(E,V)$ that has atoms. In fact, it is possible to obtain a ring of the form $\IntR(E,V)$ that is atomic. 

\begin{proposition}\label{Prop:DiscreteAtomic}
	Let $V$ be a valuation domain with principal maximal ideal. Also let $E$ be a subset of the field of fractions of $V$, and $\Gamma$ be the value group of $V$. Then $\IntR(E,V)$ is atomic if and only if $E$ is finite and $\Gamma \cong \Z$. Moreover, in this case, we have \[\IntR(E,V)^\bullet/\IntR(E,V)^\times \cong \N^{\abs{E}},\]
	where $\IntR(E,V)^\bullet$ is the multiplicative monoid of nonzero elements, making $\IntR(E,V)$ a unique factorization domain. 
\end{proposition}
\begin{proof}
	Let $\m$ denote the maximal ideal of $V$, generated by some $t \in V$.
	
	Denote by $\psi_s$ the atom in $\IntR(E,V)$ such that $v(\psi_s(d)) = 0$ for $d \in E \setminus \{s\}$ and $v(\psi_s(s)) = v(t)$. Every atom of $\IntR(E,V)$ is associate to $\psi_s$ for some $s \in E$ by Lemma \ref{Lem:AtomsInIntREV}. 
	
	Suppose that $\IntR(E,V)$ is atomic. If there exists $a \in E$ that is not an isolated point with respect to the topology induced by the valuation, then Lemma \ref{Lem:StoneWeierstrass} implies that every $\varphi \in \IntR(E,V)$ with $\varphi(a) \in \m$ is strictly divisible by a nonunit element of $\IntR(E,V)$. In particular, $t \in \IntR(E,V)$ cannot be written as the product of finitely many atoms. If on the contrary $t = \varphi_1 \cdots \varphi_n$, where $\varphi_1, \dots, \varphi_n$ are atoms of $\IntR(E,V)$, then there exists an $i$ such that $\varphi_i(a) \in \m$ for some $a \in E$ that is not an isolated point. However, we have just shown that $\varphi_i$ would not be an atom, a contradiction. Therefore, $\IntR(E,V)$ being atomic implies that $E$ consists of only isolated points. A finite product of atoms in $\IntR(E,V)$ is associates with $\psi_{s_1}\cdots \psi_{s_n}$ for some $s_1, \dots, s_n \in E$ due to Lemma \ref{Lem:AtomsInIntREV}. The rational function $\psi_{s_1}\cdots \psi_{s_n}$ is valued in $\m$ only on the finite set $\{s_1, \dots, s_n\}$. Therefore, the constant function $t \in \IntR(E,V)$ can only be written as a product of finitely many atoms if $E$ is finite. If $\Gamma \not\cong \Z$, since $\m$ is principal, there exists $c \in V$ such that there does not exist $n \in \N$ with $nv(t) = v(c)$. Therefore, if $\Gamma \not\cong \Z$, the constant function $c \in \IntR(E,V)$ cannot be written as the product of finitely many atoms of $\IntR(E,V)$. This means $\IntR(E,V)$ being atomic implies that $E$ is finite and $\Gamma \cong \Z$.
	
	Now we suppose that $E$ is finite and $\Gamma \cong \Z$. If we are able to show that $\IntR(E,V)^\bullet/\IntR(E,V)^\times \cong \N^{\abs{E}}$ as monoids, then we will show that $\IntR(E,V)$ is atomic since $\N^{\abs{E}}$ is atomic. Write $E = \{s_1, \dots, s_n\}$ for distinct elements $s_1, \dots, s_n$. The isomorphism is given by
	\[
	\varphi \mapsto (v(\varphi(s_1)), \dots, v(\varphi(s_n))),
	\]
	for any nonzero $\varphi \in \IntR(E,V)$, where $v$ is the associated valuation of $V$ with $v(t) = 1$. We see this map is indeed a well-defined injective homomorphism. The elements $\psi_{s_1}, \dots, \psi_{s_n}$, given by Lemma \ref{Lem:AtomsInIntREV}, map to a generating set of $\N^{\abs{E}}$, showing that the homomorphism is surjective. 
\end{proof}

We summarize the atomicity of the ring $\IntR(E,V)$ for a valuation domain $V$ and a subset $E$ of the field of fractions of $V$. The ring $\IntR(E,V)$ is only atomic where $E$ is finite and $V$ is a discrete valuation domain (the value group is isomorphic to $\Z$). When $\IntR(E,V)$ is atomic, it is also a unique factorization domain. If we only require $\IntR(E,V)$ to have atoms, then $V$ needs to have a principal maximal ideal and $E$ needs to have isolated points with respect to the topology induced by the valuation. Outside of this case, the ring $\IntR(E,V)$ is an antimatter domain. 

\section{Atomicity}\label{Sect:Atomicity}
\indent\indent Even for a valuation domain $V$, rings of integer-valued rational functions of the form $\IntR(E,V)$ is antimatter most of the time. Proposition \ref{Prop:DiscreteAtomic} gives an instance of a ring of the form $\IntR(E,V)$ that is atomic, but $\IntR(E,V)$ is a unique factorization domain in this case. In this section, we introduce another family of atomic rings of integer-valued rational functions. 

One way to look for atomic domains of integer-valued rational functions is to find an bounded factorization domain $D$ such that $\IntR(E,D)$ is local, which we'll see shortly that this forces $D$ to be local, for some subset $E$ of the field of fractions of $D$. Then make use of the following lemma. 

\begin{lemma}\label{Lem:LocalImpliesAtomic}
	Let $D$ be a bounded factorization domain with field of fractions $K$. Suppose $E$ is a subset of $K$ such that $\IntR(E,D)$ is local. Then $\IntR(E,D)$ is atomic. 
\end{lemma}

\begin{proof}
	We first show that if $\IntR(E,D)$ is local, then $D$ must also be local. Let $a, b \in D$ be nonunits of $D$. Then $a+b$ is not a unit in $\IntR(E,D)$ since $\IntR(E,D)$ is local, so $a+b$ is also not a unit of $D$. This shows that $\IntR(E,D)$ being local implies that $D$ is local. 
	
	Denote the maximal ideal of $D$ by $\m$. Let $a \in E$. We know that $\M_{ \m, a }$ is a maximal ideal of $\IntR(E,D)$ for each $a \in E$. Since $\IntR(E,D)$ has a unique maximal ideal, the maximal ideals $\{\M_{ \m, a }\}_{a \in E}$ all coincide, meaning the maximal ideal of $\IntR(E,D)$ is $\IntR(E,\m)$.
	
	This has the implication that for a $\varphi \in \IntR(E,D)$ such that there exists $a \in E$ where $\varphi(a) \in D^\times$, we have $\varphi(b) \in D^\times$ for all $b \in E$, meaning that $\varphi$ is a unit of $\IntR(E,D)$. 
	
	Now let $\varphi \in \IntR(E,D)$ be a nonzero, nonunit element. We want to factor $\varphi$ into a product of finitely many atoms. We induct on $\inf\{L_D(\varphi(a)) \mid a \in E\}$, which is finite since $D$ is a bounded factorization domain. 
	
	If $\inf\{L_D(\varphi(a)) \mid a \in E\} = 1$, then there exists $a \in E$ such that $\varphi(a)$ is an atom of $D$. If we write $\varphi = \psi_1\psi_2$ for some $\psi_1, \psi_2 \in \IntR(E,D)$, then $\varphi(a) = \psi_1(a)\psi_2(a)$ implies that $\psi_1(a)$ or $\psi_2(a)$ is a unit of $D$, which means that $\psi_1$ or $\psi_2$ is a unit of $\IntR(E,D)$. This shows that $\varphi$ is an atom. Since $\varphi$ is an atom, we know that in particular, $\varphi$ is a product of finitely many atoms.
	
	Now suppose that $n \coloneqq \inf\{L_D(\varphi(a)) \mid a \in E\} > 1$. Then there exists $a \in E$ such that $L_D(\varphi(a)) = n$. If $\varphi$ is an atom, we are done. If $\varphi$ is not an atom, we can write $\varphi = \psi_1\psi_2$ for some $\psi_1, \psi_2 \in \IntR(E,\m)$. Then $\varphi(a) = \psi_1(a)\psi_2(a)$. Since neither $\psi_1(a)$ nor $\psi_2(a)$ is a unit of $D$, we must have that $L_D(\psi_1(a)), L_D(\psi_2(a)) < n$. This implies that both $\inf\{L_D(\psi_1(a)) \mid a \in E\}$ and $\inf\{L_D(\psi_2(a)) \mid a \in E\}$ are strictly less than $n$. By induction, we can write $\psi_1$ and $\psi_2$ as the product of finitely many atoms of $\IntR(E,D)$, so we can do the same for $\varphi$. 
\end{proof}

One way to force a ring of integer-valued rational functions to be local is to have a local domain $D$ such that there exists a valuation overring whose valuation separates the units of $D$ and the maximal ideal of $D$. To reach this result, we use the idea of minimum valuation functions. 

\begin{definition}\cite{Liu}
	Let $V$ be a valuation domain with value group $\Gamma$, valuation $v$, and field of fractions $K$. Take a nonzero polynomial $f \in K[x]$ and write it as $f(x) = a_nx^n + \cdots + a_1x+ a_0$ for $a_0, a_1, \dots, a_n \in K$. We define the \textbf{minimum valuation function of $f$} as $\minval_{f,v}: \Gamma \to \Gamma $ by \[\gamma \mapsto \min\{v(a_0), v(a_1)+\gamma, v(a_2)+2\gamma, \dots, v(a_n) + n\gamma\}\] for each $\gamma \in \Gamma$. We will denote $\minval_{f,v}$ as $\minval_f$ if the valuation $v$ is clear from context. It is oftentimes helpful to think of $\minval_f$ as a function from $\Q{\Gamma}$ to $\Q{\Gamma}$ defined as $\gamma \mapsto \min\{v(a_0), v(a_1)+\gamma, v(a_2)+2\gamma, \dots, v(a_n) + n\gamma\}$ for each $\gamma \in \Q{\Gamma}$. 
	
	For a nonzero rational function $\varphi \in K[x]$, we write $\varphi = \frac{f}{g}$ for some $f, g \in K[x]$. Then for each $\gamma \in \Gamma$, we define $\minval_\varphi(\gamma) = \minval_f(\gamma) - \minval_g(\gamma)$.
\end{definition}

The purpose of the minimum valuation function of a rational function is that it can predict the valuation of the outputs of the rational function most of the time. The minimum valuation function also has the nice property of being piecewise linear. The following lemma showcases these facts. 

\begin{lemma}\cite[Proposition 2.24 and Lemma 2.26]{Liu}\label{Lem:MinvalForm}
	Let $V$ be a valuation domain with value group $\Gamma$, valuation $v$, maximal ideal $\m$, and field of fractions $K$. For a nonzero $\varphi \in K(x)$, the function $\minval_\varphi$ has the following form evaluated at $\gamma \in \Q\Gamma$
	\[
	\minval_\varphi(\gamma) = \begin{cases}
		c_1 \gamma + \beta_1, & \gamma \leq \delta_1,\\
		c_2 \gamma + \beta_2, & \delta_1 \leq \gamma \leq \delta_2,\\
		\vdots\\
		c_{k-1} \gamma + \beta_{k-1}, & \delta_{k-2} \leq \gamma  \leq \delta_{k-1},\\
		c_k \gamma + \beta_k, & \delta_{k-1} \leq \gamma,
	\end{cases}
	\] 
	where $c_1, \dots, c_k \in \Z$; $\beta_1, \dots, \beta_k \in \Gamma$; and $\delta_1, \dots, \delta_{k-1} \in \Q{\Gamma}$ such that $\delta_1 < \cdots < \delta_{k-1}$. 
	
	Furthermore, for all but finitely many $\gamma \in \Gamma$, we have that $v(\varphi(t)) = \minval_\varphi(v(t))$ for all $t \in K$ such that $v(t) = \gamma$. 
\end{lemma}

Let $D$ be a local domain with field of fractions $K$. We now develop a technical lemma using the minimum valuation function to say that if all of the nonzero rational functions in $\IntR(K,D)$ have a minimum valuation function that is identically $0$ or always strictly positive, then $\IntR(K,D)$ is a local domain.

\begin{lemma}\label{Lem:MinvalDichotomyImpliesLocal}
	Let $D$ be a local domain with maximal ideal $\m$ and field of fractions $K$. Suppose that there is a valuation overring $V$ of $D$ with valuation $v$ such that for all $d \in \m$, we have $v(d) > 0$. Also let $\Gamma$ be the value group of $v$. Suppose that for all nonzero $\varphi \in \IntR(K,D)$ that $\minval_\varphi(\gamma) = 0$ for all $\gamma \in \Gamma$ or $\minval_\varphi(\gamma) > 0$ for all $\gamma \in \Gamma$. Then $\IntR(K,D)$ is a local domain with maximal ideal $\IntR(K,\m)$.
\end{lemma}

\begin{proof}
	Let $\varphi \in \IntR(K,D)$. Suppose that $\minval_\varphi(\gamma) = 0$ for all $\gamma \in \Gamma$. We want to show that $\varphi$ is a unit of $\IntR(K,D)$. Assume for a contradiction that there exists $a \in K$ such that $v(\varphi(a)) > 0$. We know that $a \neq 0$ because otherwise, we have $v(\varphi(d)) > 0$ for all $d \in K$ with $v(d)$ sufficiently large by Proposition \ref{Prop:RationalFunctionsContinuousInValuationTopology} and this contradicts the assumption that $\minval_\varphi(\gamma) = 0$ for all $\gamma \in \Gamma$ by Lemma \ref{Lem:MinvalForm}. 
	
	Now assume that there exists $a \in K$ with $a \neq 0$ such that $v(\varphi(a)) > 0$. Then form 
	\[
	\psi(x) :=\varphi(a(1+x)) \in \IntR(K,D).
	\]
	Since $v(\psi(a)) = v(\varphi(a)) > 0$, by Proposition \ref{Prop:RationalFunctionsContinuousInValuationTopology}, there exists $\delta \in \Gamma$ such that for all $d \in K$ with $v(d) \geq \delta$, we have $v(\psi(d)) > 0$. By Lemma \ref{Lem:MinvalForm}, we can make $\delta$ large enough so that $\minval_\psi(\gamma) > 0$ for all $\gamma \in \Gamma$ with $\gamma \geq \delta$. However, there also exists $\alpha \in \Gamma$ with $\alpha < 0$ such that $\minval_\psi(\alpha) = v(\psi(b))$ for all $b \in K$ with $v(b) = \alpha$ and $\minval_\varphi(v(a)+\alpha) = v(\varphi(c))$ for all $c \in K$ with $v(c) = v(a) + \alpha$ by Lemma \ref{Lem:MinvalForm}. Let $b \in K$ such that $v(b) = \alpha$. Then
	\[
	\minval_\psi(\alpha) = v(\psi(b)) = v(\varphi(a(1+b))) = \minval_\varphi(v(a)+\alpha) = 0. 
	\]
	This gives us $\minval_\psi(\delta) > 0$ but $\minval_{\psi}(\alpha) = 0$, which is disallowed by assumption. Thus, for all $a \in K$, we have $v(\varphi(a)) = 0$ and $\varphi(a) \in D$ so $\varphi(a) \in D^\times$. 
	
	Now if we assume that $\varphi \in \IntR(K,D)$ is such that $\minval_\varphi(\gamma) > 0$ for all $\gamma \in \Gamma$, then we can similarly conclude that for all $a \in K$, we have $v(\varphi(a)) > 0$, so $\varphi(a) \in \m$. 
	
	Therefore, $\IntR(K,D) = \IntR(K,D)^\times \sqcup \IntR(K,\m)$, meaning that $\IntR(K,D)$ is a local domain with maximal ideal $\IntR(K,\m)$. 
\end{proof}

Next, we can construct local domains $D$ that satisfy the previous lemma by requiring the valuation on $D$ to give every element of the maximal ideal of $D$ a strictly positive value.

\begin{lemma}\label{Lem:GapImpliesLocal}
	Let $D$ be a local domain with maximal ideal $\m$ and field of fractions $K$. Suppose there exists a valuation overring $V$ of $D$ such that maximal ideal of $V$ is not principal. Let $\Gamma$ be the value group of $V$. Suppose that there exists $\gamma \in \Gamma$ with $\gamma > 0$ such that $v(d) > \gamma$ for each $d \in \m$. Then $\IntR(K,D)$ is a local domain with maximal ideal $\IntR(K,\m)$. 
\end{lemma}

\begin{proof}
	Let $\varphi \in \IntR(K,D)$ be nonzero. By Lemma \ref{Lem:MinvalForm}, we have that $\minval_\varphi$ is a piecewise linear function that maps $\Gamma$ to $\{0\} \sqcup v(\m)$. Since the maximal ideal of $V$ is not principal, the value group $\Gamma$ is not discrete. Thus, we must have $\minval_\varphi = 0$ or $\minval_\varphi(\gamma) > 0$ for all $\gamma \in \Gamma$. Then $\IntR(K,D)$ is a local domain with maximal ideal $\IntR(K,\m)$ by Lemma \ref{Lem:MinvalDichotomyImpliesLocal}.
	
\end{proof}

We still need the base ring to be a bounded factorization domain in order to use Lemma \ref{Lem:LocalImpliesAtomic}. We get this if the image of the base ring under the separating valuation is a bounded factorization monoid. 

\begin{lemma}\label{Lem:BFMtoLocalAtomic}
	Let $D$ be a local domain with maximal ideal $\m$ and field of fractions $K$. Suppose there exists a valuation overring $V$ of $D$ such that maximal ideal of $V$ is not principal. Let $\Gamma$ be the value group of $V$ and $v$ be the associated valuation. Suppose that there exists $\gamma \in \Gamma$ with $\gamma > 0$ such that $v(d) > \gamma$ for each $d \in \m$ and $M \coloneqq \{0\} \cup v(\m)$ is a bounded factorization monoid. Then $\IntR(K,D)$ is a local, atomic domain. 
\end{lemma}

\begin{proof}
	We want to use Lemma \ref{Lem:LocalImpliesAtomic}, so we will show that $D$ is a bounded factorization domain. 
	
	Let $d \in \m \setminus \{0\}$. We want to show that we can write $d$ as a product of finitely many atoms of $D$ by inducting on $L_M(v(d))$. We know that $L_M(v(d))$ is finite since $M$ is a bounded factorization monoid. 
	
	If $L_M(v(d)) = 1$, then $v(d)$ is an atom of $M$. Whenever we write $d = b_1b_2$ for some $b_1, b_2 \in D$, we have $v(d) = v(b_1) + v(b_2)$. Thus, $v(b_1) = 0$ or $v(b_2) = 0$, so $b_1$ or $b_2$ is a unit of $D$ by the assumption that there exists $\gamma \in \Gamma$ with $\gamma > 0$ such that $v(d) > \gamma$ for each $d \in \m$. Therefore, $d$ is an atom of $D$.
	
	Suppose that $n \coloneqq L_M(v(d))  > 1$. If $d$ is an atom of $D$, there is nothing to do. If $d$ is not an atom of $D$, then we can write $d = b_1b_2$ for some $b_1, b_2 \in \m$. Since $v(d) = v(b_1) + v(b_2)$, we must have that $L_M(v(b_1)), L_M(v(b_2)) < n$. By induction, we know that we can write $b_1$ and $b_2$ both as a product of finitely many atoms of $D$, so $d$ can be written as the product of finitely many atoms of $D$ as well. 
	
	Now write $d = a_1 \cdots a_m$ for some atoms $a_1, \dots, a_m \in D$. Then we see that $v(d) = v(a_1) + \cdots + v(a_m)$ with no $v(a_i) = 0$ implies that $m \leq L_M(v(d)) < \infty$. Thus, $D$ is a bounded factorization domain. 
	
	We now use Lemma \ref{Lem:GapImpliesLocal} to show that $\IntR(K,D)$ is local, so Lemma \ref{Lem:LocalImpliesAtomic} tells us that $\IntR(K,D)$ is atomic. 
\end{proof}

We will provide examples of domains $D$ that satisfy the conditions of the previous lemma by starting with the monoid $M$. We can form a ring out of this monoid called the monoid ring. A collection of results on these rings can be found in \cite{GilmerSemigroupRings}. We then localize these monoid rings to satisfy the local condition. 

\begin{definition}
	Let $R$ be a domain and $M$ be an commutative, cancellative, additive monoid. We form the \textbf{monoid ring} $R[t;M]$ whose elements are finite sums of the form $rt^{\alpha}$, where $r \in R$ and $\alpha \in M$. The multiplication in $R[t;M]$ is determined by $r_1t^{\alpha_1} \cdot r_2t^{\alpha_2} = r_1r_2t^{\alpha_1+\alpha_2}$ for $r_1, r_2 \in R$ and $\alpha_1, \alpha_2 \in M$. Denote by $(t;M)$ the ideal of $R[t;M]$ generated by $\{t^\alpha \mid \alpha \text{ is a nonzero nonunit of $M$}\}$.  
\end{definition}

The ring $\F_2[t;M]_{(t;M)}$, where $M = \{q \in \Q \mid q = 0 \text{ or } q \geq 1\}$ has been used to produce a SFT ring of finite Krull dimension such that its power series ring has infinite Krull dimension \cite{CoykendallSFT}. Here, we will use localized monoid domains to produce a family of local, atomic rings of integer-valued rational functions. If the underlying monoid is totally ordered, these localized monoid domains have a useful valuation on them to produce the valuation domain required for Lemma \ref{Lem:BFMtoLocalAtomic}. 

\begin{definition}
	Let $M$ be a monoid. The monoid $M$ is \textbf{totally ordered} if there is a total order $\leq$ on $M$ such that for all $\alpha, \alpha', \beta, \beta' \in M$, having $\alpha \leq \alpha'$ and $\beta \leq \beta'$ imply $\alpha + \beta \leq \alpha' + \beta'$. 
	
	Let $M$ be a totally ordered monoid. We say that $M$ is \textbf{positive} if $ \alpha \geq 0$ for all $\alpha \in M$. The \textbf{difference group} of $M$ is $\mathsf{gp}(M) \coloneqq \{\alpha - \alpha' \mid \alpha, \alpha' \in M \}$. Take $k$ to be a field and $M$ to be a totally ordered positive monoid, and set $D$ to be $k[t;M]_{(t;M)}$. Let $K$ be the field of fractions of $D$. The \textbf{$M$-valuation} $v$ on $D$ is given by
	\[
	v\left( a_{\alpha_{i_1}} t^{\alpha_{i_1}} + \cdots + a_{\alpha_{i_n}}t^{\alpha_{i_n}} \right) = \min\{\alpha_{i_1}, \dots, \alpha_{i_n} \},
	\]
	where $\alpha_{i_1}, \dots, \alpha_{i_n} \in M$ are distinct and $a_{\alpha_{i_1}}, \dots, a_{\alpha_{i_n}} \in k$ are nonzero elements. This valuation on $k[t;M]$ extends uniquely to a valuation on $D$. The valuation extends further uniquely to a valuation on $K$ with value group $\mathsf{gp}(M)$. 
	
	For a totally ordered, positive monoid $M$, we say that $M$ is \textbf{bounded away from $0$} if there exists $\gamma \in \mathsf{gp}(M)$ with $\gamma > 0$ such that $\alpha \geq \gamma$ for each nonzero $\alpha \in M$. We say that $M$ is \textbf{integrally terminal for $\gamma$} for some $\gamma \in \mathsf{gp}(M)$ if for all $\gamma' \in \mathsf{gp}(M)$ such that $\gamma' > \gamma$, we have $\gamma' \in M$ and for each $\gamma' \in \mathsf{gp}(M)$ with $\gamma' > 0$, there exists $n \in \N$ such that $n\gamma' > \gamma$. Note that the latter condition implies that $n\gamma' \in M$. We say that $M$ is \textbf{integrally terminal} if there exists some $\gamma \in \mathsf{gp}(M)$ such that $M$ is integrally terminal for $\gamma$.
\end{definition}

\begin{theorem}\label{Thm:Atomic}
	Let $M$ be a totally ordered, positive, bounded factorization monoid bounded away from $0$. Suppose that $\Gamma \coloneqq \mathsf{gp}(M)$ has no minimal strictly positive element. Let $k$ be a field and take $D$ to be $k[t;M]_{(t;M)}$. Also let $K$ be the field of fractions of $D$. Then $\IntR(K,D)$ is a local, atomic domain. 
\end{theorem}

\begin{proof}
	Note that since $M$ is bounded away from $0$, there exists $\gamma \in \Gamma$ with $\gamma > 0$ such that $\alpha \geq \gamma$ for each $\alpha \in M\setminus \{0\}$.
	
	Let $v$ be the $M$-valuation on $K$. Call the associated valuation domain $V$. Since $\Gamma$ has no minimal strictly positive element, the maximal ideal of $V$ is not principal. Plus, $\{0\} \cup v((t;M)k[t;M]_{(t;M)}) = M$ is a bounded factorization monoid. Furthermore, we have that for each $d \in (t;M)k[t;M]_{(t;M)}$ that $v(d) \in M$ so $v(d) > \gamma$. By Lemma \ref{Lem:BFMtoLocalAtomic}, we have that $\IntR(K,D)$ is a local, atomic domain. 
\end{proof}

In some cases, the atomicity of these local, atomic domains of the form $\IntR(K,D)$, where $D = k[t;M]_{(t;M)}$, is fragile in the sense that the integral closure of $\IntR(K,D)$ is an antimatter domain, which we will see in Proposition \ref{Prop:IntegralClosureAntimatter}. In particular, this happens when we additionally require $M$ to be integrally terminal in addition to the conditions in Theorem \ref{Thm:Atomic}. The notion of $M$ being integrally terminal can be thought of as $M$ containing all sufficiently large elements of $\mathsf{gp}(M)$ and this sufficiently large portion of $\mathsf{gp}(M)$ is most of $\mathsf{gp}(M)$. Before we show that the integral closure of $\IntR(K,D)$ is antimatter in this case, we give a criterion for elements of $K$ to be in $D$ using the $M$-valuation. 

\begin{lemma}\label{Lem:BiggestFilter}
	Let $M$ be a totally ordered, positive monoid that is integrally terminal for some $\delta \in \Gamma \coloneqq \mathsf{gp}(M)$. Let $k$ be a field and take $D$ to be $k[t;M]_{(t;M)}$. Let $K$ be the field of fractions of $D$ and let $v$ be the $M$-valuation on $K$. If $c \in K$ is such that $v(c) > \delta$, or if $v(c) = \delta$ and $\delta \in M$, then $c \in D$. 
\end{lemma}

\begin{proof}
	Let $c \in K$ such that $v(c) > \delta$. We can write $c$ as
	\[
	c =  \frac{\sum\limits_{\alpha \in \Gamma_{\geq 0}} a_{\alpha} t^\alpha}{\sum\limits_{\beta \in \Gamma_{\geq 0}} b_\beta t^\beta},
	\]
	where each $a_\alpha, b_\beta \in k$ and $a_\alpha, b_\beta = 0$ for all but finitely many $\alpha, \beta \in \Gamma_{\geq 0}$. 	We write $c$ so that $b_0 \neq 0$ and $\min\{\alpha \in \Gamma_{\geq 0} \mid a_\alpha \neq 0 \} = v(c)$. If $\{\beta \in \Gamma_{\geq 0} \setminus \{0\} \mid b_\beta \neq 0 \}$ is empty, then $c \in k[t;M] \subseteq D$ and we are done. Otherwise, let $\varepsilon = \min\{\beta \in \Gamma_{\geq 0} \setminus \{0\} \mid b_\beta \neq 0 \}$.  Now there is some odd $n\in \N$ such that $n\varepsilon > \delta$. We have that $x+b_0$ divides $x^n+b_0^n$ in $k[x]$. Let $f(x) = \frac{x^n+b_0^n}{x+b_0} \in k[x]$. Then we have
	\[
	c = \frac{f\left(-b_0 + \sum\limits_{\beta \in \Gamma_{\geq 0}} b_\beta t^\beta \right)\sum\limits_{\alpha \in \Gamma_{\geq 0}} a_{\alpha} t^\alpha}{f\left(-b_0 + \sum\limits_{\beta \in \Gamma_{\geq 0}} b_\beta t^\beta \right)\sum\limits_{\beta \in \Gamma_{\geq 0}} b_\beta t^\beta} = \frac{f\left(-b_0 + \sum\limits_{\beta \in \Gamma_{\geq 0}} b_\beta t^\beta \right)\sum\limits_{\alpha \in \Gamma_{\geq 0}} a_{\alpha} t^\alpha}{\left(-b_0 + \sum\limits_{\beta \in \Gamma_{\geq 0}} b_\beta t^\beta \right)^n + b_0^n}
	\] 
	Each term in $\left(\sum\limits_{\beta \in \Gamma_{\geq 0} \setminus \{0\}} b_\beta t^\beta\right)^n$ has value at least $n\varepsilon > \delta$. Therefore, the entire denominator $\left(\sum\limits_{\beta \in \Gamma_{\geq 0} \setminus \{0\}} b_\beta t^\beta\right)^n + b_0^n$ is in $k[t;M] \setminus (t;M)$. As for the numerator, each term when expanded out has valuation at least $v(c) > \delta$, so the numerator is in $k[t;M]$. Thus, $c \in D$. 
	
	If $v(c) = \delta$ and $\delta \in M$. The proof is similar. 
	
\end{proof}

\begin{example}
	Let $M$ be a positive submonoid of $\R_{\geq 0}$ such that there exists $\delta \in \R_{\geq 0}$ so that $[\delta, \infty) \subseteq M$. Note that $M$ is integrally terminal. Also let $k$ be a field and $D = k[t;M]_{(t;M)}$. Denote by $K$ the field of fractions of $D$ and $v$ the $M$-valuation on $K$. Then $\frac{t^{13\delta} + t^{7\delta} + t^{3\delta}}{t^{\delta/2}+1} \in D$ since $v\left(\frac{t^{13\delta} + t^{7\delta} + t^{3\delta}}{t^{\delta/2}+1}\right) = 3\delta > \delta$.
\end{example}

Now we compute the integral closure of $\IntR(K,D)$ under certain assumptions and prove that this ring is an antimatter domain. 

\begin{proposition}\label{Prop:IntegralClosureAntimatter}
	Let $M$ be a totally ordered, integrally terminal, positive monoid bounded away from $0$. Suppose that $\Gamma \coloneqq \mathsf{gp}(M)$ has no minimal strictly positive element. Let $k$ be a field and take $D$ to be the ring $k[t;M]_{(t;M)}$ and let $K$ be the field of fractions of $D$. Denote by $V$ the valuation domain associated with the $M$-valuation $v$ on $K$. The integral closure of $\IntR(K,D)$ is the local, antimatter domain $\IntR(K,D)^\times + \M$, where
	\[
	\M = \{\varphi \in \IntR(K,V) \mid \exists \gamma \in \Gamma_{> 0} \text{ such that } \,\forall a \in K, v(\varphi(a)) \geq \gamma \}
	\]
	and is the maximal ideal of the integral closure of $\IntR(K,D)$. 
\end{proposition}

\begin{proof}
	Since $M$ is bounded away from $0$, there exists some $\gamma_1 \in \Gamma_{> 0}$ such that $\alpha \geq \gamma_1$ for all nonzero $\alpha \in M$. 
	
	Since $M$ is integrally terminal, there exists some $\gamma_2 \in \Gamma$ such that for all $\gamma \in \Gamma$ such that $\gamma > \gamma_2$, we have that $\gamma \in M$ and that for each $\gamma \in \Gamma$ such that $\gamma > 0$, there exists $n \in \N$ such that $n\gamma > \gamma_2$.  
	
	Let $\IntR(K,D)'$ be the integral closure of $\IntR(K,D)$. First, we notice that $\IntR(K,D)^\times \subseteq \IntR(K,D) \subseteq \IntR(K,D)'$. If we take $\varphi \in \M$, then there exists $\gamma \in \Gamma$ with $\gamma > 0$ such that for all $a \in K$, we have $v(\varphi(a)) \geq \gamma$. Then there exists $n \in \N$ with $n > 0$ such that $n \gamma > \gamma_2$. Thus, $v(\varphi^n(a)) \geq n\gamma > \gamma_2$ for all $a \in K$. Using this, we can see that $\varphi^n \in \IntR(K,D)$ so $\varphi \in \IntR(K,D)'$. Thus, $\IntR(K,D)'$ contains $\IntR(K,D)^\times + \M$. 
	
	Now let $\varphi \in \IntR(K,D)'$. Then $\varphi^n + \psi_{n-1}\varphi^{n-1} + \cdots + \psi_1 \varphi + \psi_0 = 0$ for some $\psi_0, \dots, \psi_{n-1} \in \IntR(K,D)$. Then for every $a \in K$, we see that $\varphi(a)$ satisfies some monic polynomial with coefficients in $D$, so $\varphi(a)$ is in the integral closure of $D$, which is contained in $V$. Thus, $\varphi \in \IntR(K,V)$. We consider the case of when $\varphi$ is not a unit of $\IntR(K,V)$ and the case of when $\varphi$ is a unit of $\IntR(K,V)$.
	
	First, consider the case where $\varphi \notin \IntR(K,V)^\times$. Suppose that there does not exist $\gamma \in \Gamma$ with $\gamma > 0$ such that $v(\varphi(a)) \geq \gamma$ for all $a \in K$. We look at the monic polynomial over $\IntR(K,D)$ that $\varphi$ satisfies: $\psi_{n}\varphi^n + \psi_{n-1}\varphi^{n-1} + \cdots + \psi_1 \varphi + \psi_0 = 0$ for some $\psi_0, \dots, \psi_{n-1} \in \IntR(K,D)$ and $\psi_{n} = 1$. There exists some $i \in \{ 1, \dots, n\}$ such that $\psi_{1}, \dots, \psi_{i-1}$ are all in the maximal ideal of $\IntR(K,D)$ and $\psi_i \in \IntR(K,D)^\times$. We claim that there exists $a \in K$ such that $0 < v(\varphi(a)) < \frac{\gamma_1}{i}$. By Lemma \ref{Lem:MinvalForm}, if $\minval_\varphi$ attains a value of $0$ and a value that is at least $\frac{\gamma_1}{i}$, then there exists an $a \in K$ such that $0 < v(\varphi(a)) < \frac{\gamma_1}{i}$. Now we consider either $\minval_\varphi = 0$ or $\minval_\varphi(\gamma) \geq \frac{\gamma_1}{i}$ for all $\gamma \in \Gamma$. Using the proof of Lemma \ref{Lem:MinvalDichotomyImpliesLocal}, we see that $\minval_\varphi = 0$ implies that $\varphi \in \IntR(K,V)^\times$ and that $\minval_\varphi(\gamma) \geq \frac{\gamma_1}{i}$ for all $\gamma \in \Gamma$ implies $v(\varphi(d)) \geq \frac{\gamma_1}{i}$ for all $d \in K$. Both cases do not fall within our assumptions, so we must have an $a \in K$ such that $0 < v(\varphi(a)) < \frac{\gamma_1}{i}$. We get that
	\[
	v(\psi_0(a)) = v(\varphi(a)^n + \psi_{n-1}(a)\varphi(a)^{n-1} + \cdots + \psi_1(a)\varphi(a) ) = iv(\varphi(a)) < \gamma_1,
	\]
	since we have $v(\psi_j(a)\varphi(a)^j) = v(\psi_j(a)) + jv(\varphi(a)) \geq \gamma_1 > iv(\varphi(a))$ for $j$ such that $1 \leq j < i$ and we have $v(\psi_j(a)\varphi(a)^j) \geq jv(\varphi(a)) > iv(\varphi(a))$ for $j$ such that $i < j \leq n$. On the other hand,
	\begin{align*}
		v(\psi_0(a)) & = v\left(\sum_{j=1}^{n}\psi_j(a)\varphi(a)^j \right)
		\\
		& = v(\varphi(a)) + v\left(\sum_{j=1}^{n}\psi_j(a)\varphi(a)^{j-1}\right) \\
		& \geq v(\varphi(a)) > 0		
	\end{align*}
	This implies that $0 < v(\psi_0(a)) < \gamma_1$, which is impossible since $\psi_0(a) \in D$. This means that there must exist a $\gamma \in \Gamma$ with $\gamma > 0$ such that $v(\varphi(a)) \geq \gamma$ for all $a \in K$. Therefore, $\varphi \in \M$. 
	
	If $\varphi \in \IntR(K,V)^\times$, we have that $v(\varphi(0) - c) > 0$ for some $c \in k^\times$. Then set $\Phi(x):= \varphi(x) - c$. Now, $\Phi(x) \notin \IntR(K,V)^\times$ since $v(\Phi(0)) > 0$. We know $\Phi \in \IntR(K,D)'$, so we have $\Phi \in \M$ by the previous case. Therefore, we conclude that $\varphi \in c + \M \subseteq \IntR(K,D)^\times + \M$.

	Now that $\IntR(K,D)' = \IntR(K,D)^\times + \M$, we see that the set of nonunits of $\IntR(K,D)'$ is $\M$ and $\M$ is closed under addition. Thus, $\IntR(K,D)'$ is local with maximal ideal $\M$. 
	
	To see that $\IntR(K,D)'$ is antimatter, we take a nonzero $\varphi \in \M$. There is some $\gamma \in \Gamma$ with $\gamma > 0$ such that $v(\varphi(a)) \geq \gamma$ for all $a \in K$. There exists some $\gamma' \in \Gamma$ such that $0 < \gamma' < \gamma$ since $\Gamma$ has no minimal strictly positive element. Then $\varphi = t^{\gamma'}\frac{\varphi}{t^{\gamma'}}$ with $t^{\gamma'}, \frac{\varphi}{t^{\gamma'}} \in \M$. Thus, any nonzero, nonunit element of $\IntR(K,D)'$ is strictly divisible by a nonunit element. 
\end{proof}

\section{Factorization lengths}\label{Sect:Lengths}

\indent\indent For the atomic family of integer-valued rational functions over localizations of monoid domains in the previous section, the factorization of constant rational functions is different than that of the same constant considered as an element of the base ring. One way this difference in factorization behavior is captured is by comparing the set of factorization lengths. We first prove that atoms in the base ring remain atoms in the ring of integer-valued rational functions. 

\begin{lemma}\label{Lem:AtomTransfer}
	Let $D$ be an atomic domain with field of fractions $K$. Let $E$ be a nonempty subset of $K$ such that $\IntR(E,D)$ is a local domain. If $d \in D$ is an atom of $D$, then $d$ is an atom of $\IntR(E,D)$. 
\end{lemma}

\begin{proof}
	We have seen in the proof of Lemma $\ref{Lem:LocalImpliesAtomic}$ that $\IntR(E,D)$ being local implies that $D$ is local. Furthermore, $\IntR(E,\m)$ is the unique maximal ideal of $\IntR(E,D)$, where $\m$ is the unique maximal ideal of $D$. 
	
	Let $d \in D$ be an atom of $D$. Suppose $d = \varphi_1\varphi_2$ for some $\varphi_1, \varphi_2 \in \IntR(E,D)$. Then $d = \varphi_1(a)\varphi_2(a)$ for all $a \in E$. If we fix an $a \in E$, we see that $\varphi_1(a)$ is a unit or $\varphi_2(a)$ is a unit of $D$. Since the maximal ideal of $\IntR(E,D)$ is $\IntR(E,\m)$, we know that $\varphi_1$ or $\varphi_2$ is a unit of $\IntR(E,D)$. This implies that $d$ is an atom of $\IntR(E,D)$. 
\end{proof}

Let $k$ be a field and $M$ be a totally ordered, positive monoid. Take $D$ to be $k[t;M]_{(t;M)}$ and $K$ to be its field of fractions. Lemma \ref{Lem:BiggestFilter} shows that if $M$ is integrally terminal, an element of $K$ with sufficiently large value is in $D$. The part of $M$ that is not sufficiently large does not behave as nicely. For example, if an integer-valued rational function in $\IntR(K,D)$ has some $a \in K$ such that $v(\varphi(a))$ is not part of the sufficiently large part of $M$, then $\varphi$ is restricted in its form.

\begin{lemma}\label{Lem:Flat}
	Let $k$ be a field and $M$ be a totally ordered, positive monoid. Set $\Gamma \coloneqq \mathsf{gp}(M)$ and suppose $\Gamma$ has no minimal strictly positive element. Let $D = k[t;M]_{(t;M)}$, $K$ be its field of fractions, and $v$ the $M$-valuation on $K$. Take some $\varphi \in \IntR(K,D)$. Suppose there exist $b \in K$ and $\delta, \delta' \in \Gamma_{\geq 0}$ such that $0 < v(\varphi(b)) < \delta < \delta'$ and there does not exist $\alpha \in M$ such that $\delta < \alpha < \delta'$. Then $\{v(\varphi(a)) \mid a \in K \} = \{v(\varphi(b))\}$. 
\end{lemma}

\begin{proof}
	Without loss of generality, we can assume that $b = 0$. Consider the function $\psi(x) \coloneqq \varphi(x) - \varphi(0)$. Since $\psi(0) = 0$, when we write $\psi(x) = \frac{f(x)}{g(x)}$ for some $f, g \in K[x]$ coprime, we have that $\minval_g(\gamma) = v(g(0)) < \infty$ for sufficiently large $\gamma \in \Gamma$ and $\minval_f(\gamma)$ can be arbitrarily large because $x$ divides $f$ in $K[x]$. This implies that $\minval_\psi(\gamma) = \minval_f(\gamma) - v(g(0)) \geq \delta'$ for all sufficiently large values of $\gamma \in \Gamma$. Since $\minval_\psi$ has image contained in $M$, which does not contain the interval from $\delta$ to $\delta'$, exclusive, and Lemma \ref{Lem:MinvalForm} says that $\minval_\psi$ is piecewise linear, we know that $\minval_\psi(\gamma) \geq \delta'$ for all $\gamma \in \Gamma$. 
	
	Suppose there exists $a \in K$ such that $v(\psi(a)) < \delta'$. Consider $\rho(x) \coloneqq \psi(a(1+x))$. By Lemma \ref{Lem:MinvalForm}, for all but finitely many values of $\gamma \in \Gamma$ with $\gamma < 0$, we have $\minval_\rho(\gamma) = v(\rho(d))$ and $\minval_\psi(\gamma) = v(\psi(d))$ for all $d \in K$ such that $v(d) = \gamma$. In particular, we can find a $\gamma \in \Gamma$ with $\gamma < 0$ such that $\minval_\rho(\gamma) = v(\rho(d))$ and $\minval_\psi(v(a) + \gamma) = v(\psi(ad'))$ for all $d, d' \in K$ such that $v(d) = v(d') = \gamma$. This means that there exists $d \in K$ with $v(d) < 0$ such that $\minval_\rho(v(d)) = v(\rho(d))$ and $\minval_\psi(v(a) + v(d)) = v(\psi(a(1+d)))$. Then
	\[
		\minval_\rho(v(d)) = v(\rho(d)) = v(\psi(a(1+d))) = \minval_\psi(v(a) + v(d)) \geq \delta'.
	\]  
	However, $v(\rho(0)) = v(\psi(a)) < \delta$. Writing $\rho(x) = \frac{F(x)}{G(x)}$ for $F, G \in K[x]$ coprime allows us to see that $\minval_F(\gamma) = v(F(0))$ and $\minval_G(\gamma) = v(G(0))$ for sufficiently large $\gamma$, so $\minval_\rho(\gamma) = v(F(0)) - v(G(0)) = v(\rho(0)) < \delta$ for $\gamma \in \Gamma$ sufficiently large. This is a contradiction since $\minval_\rho$ cannot attain both values greater than or equal to $\delta'$ and strictly less than $\delta$ due to Lemma \ref{Lem:MinvalForm} and the fact that $M$ does not contain any values strictly between $\delta$ and $\delta'$. Thus, $v(\psi(a)) \geq \delta'$ for all $a \in K$. 
	
	Now we see that $v(\varphi(a) - \varphi(0)) \geq \delta'$ for all $a \in K$. Since $v(\varphi(0)) < \delta'$, we have that $v(\varphi(a)) = v(\varphi(0))$ for all $a \in K$. 

\end{proof}

\begin{example}
	Let $M = \{0\} \cup [2,3] \cup [4, \infty)$ be a submonoid of $\R$ and $k$ to be any field. Take $D$ to be $k[t;M]_{(t;M)}$ and $K$ its field of fractions. Let $v$ be the $M$-valuation on $K$. Take $\varphi \in \IntR(K,D)$. Suppose that there exists $b \in K$ such that $v(\varphi(b)) = 2$. Since the interval $(3,4)$ is missing from $M$, then we must have $v(\varphi(a)) = 2$ for all $a \in K$. 
\end{example}

The previous example also satisfies the following theorem if $k$ is not algebraically closed. Another example that also satisfies the following theorem is if we replaced $M$ with $\left(\{0,1,\sqrt{2}\} \cup ([2\sqrt{2}-1,\infty)  \right)\cap \Z[\sqrt{2}]$. The field $k$ can be algebraically closed in this case and still satisfy the conditions of the theorem since $\mathsf{gp}(M) = \Z[\sqrt{2}]$ is not divisible. 

\begin{theorem}\label{Thm:ExtendSetOfLengths}
	Let $M$ be a totally ordered, integrally terminal, positive, bounded factorization monoid. Suppose that $\Gamma \coloneqq \mathsf{gp}(M)$ has no minimal strictly positive element. Further suppose that there exists some $\alpha \in M$ for which $M$ is integrally terminal and some $\beta \in \Gamma$ with $0 < \beta < \alpha$ such that no $\gamma \in M$ has the property $\beta < \gamma < \alpha$. Let $k$ be a field and suppose that $k$ is not algebraically closed or $\Gamma$ is not divisible. Take $D = k[t;M]_{(t;M)}$, $K$ to be its field of fractions, and $v$ to be the $M$-valuation on $K$. Then for any nonzero, nonunit $d \in D$ with $v(d) > 3\alpha$, we have
	\[
	\mathcal{L}_D(d) \cup \{2, \dots, N-2 \} \subseteq \mathcal{L}_{\IntR(K,D)}(d) \subseteq \{2, \dots, L_D(d)\},
	\]
	where $N \coloneqq \min\{n\in \N \mid n\alpha > v(d)\}$. 
\end{theorem}

\begin{proof}
	We observe that from Lemma \ref{Lem:AtomTransfer}, we have that atoms of $D$ are atoms of $\IntR(K,D)$. Thus, $\mathcal{L}_D(d) \subseteq \mathcal{L}_{\IntR(K,D)}(d)$. 
	
	The existence of $\beta \in \Gamma$ with $0 < \beta < \alpha$ such that no $\gamma \in M$ has the property $\beta < \gamma < \alpha$ means that there does not exist $\gamma \in M$ such that $0 < \gamma < \alpha - \beta$. This implies that $M$ is bounded away from $0$. Since $M$ is bounded away from $0$, there exists $\gamma_1 \in \Gamma$ with $\gamma_1 > 0$ such that for all nonzero $\gamma \in M$, we have that $\gamma \geq \gamma_1$. 
	
	We now split into two cases. For the first case, we assume that $k$ is not algebraically closed and $\Gamma$ is divisible. For the second case, we assume the $\Gamma$ is not divisible.
	
	\begin{enumerate}
		\item Assume that $k$ is not algebraically closed and $\Gamma$ is divisible. We now fix an $\ell \in \{2, \dots, N-2\}$ and set $\alpha' = \frac{v(d)-\alpha}{\ell - 1}$. Then we write
		\[
		d = \underbrace{\varphi_{\alpha, \alpha',\varepsilon} \cdots \varphi_{\alpha, \alpha',\varepsilon}}_{\text{$\ell - 1$ times}} \cdot \frac{d}{\varphi_{\alpha, \alpha', \varepsilon}^{\ell-1}},
		\]
		where $\varphi_{\alpha, \alpha', \varepsilon}\in K(x)$ is such that there exist $\delta, \delta' \in \Q\Gamma$ so that for all $a \in K$, we have 	
		\[
		v(\varphi_{\alpha, \alpha', \varepsilon}(a)) = \begin{cases}
			\alpha, & \text{if } v(a) < \delta, \\
			\alpha', & \text{if } v(a) > \delta',
		\end{cases}
		\]
		and $\alpha \leq  v(\varphi_{\alpha, \alpha', \varepsilon}(a)) \leq \alpha'$ if $\delta \leq v(d) \leq \delta'$, guaranteed by Lemma \ref{Lem:ZigZag}, using any value for $\varepsilon \in \Q\Gamma$ such that $\varepsilon > 0$. We will show that each factor is an atom of $\IntR(K,D)$. We have that $\varphi_{\alpha, \alpha', \varepsilon} \in \IntR(K,D)$ by Lemma \ref{Lem:BiggestFilter}. Suppose that $\varphi_{\alpha, \alpha', \varepsilon} = \psi_1\psi_2$ for some $\psi_1, \psi_2 \in \IntR(K,D)$. Take $a \in K$ with $v(a) < 0$. Then $\alpha = v(\varphi_{\alpha, \alpha'}(a)) = v(\psi_1(a)) + v(\psi_2(a))$. If $v(\psi_1(a))$ or $v(\psi_2(a))$ is $0$, then $\psi_1$ or $\psi_2$ is a unit of $\IntR(K,D)$ by Theorem \ref{Thm:Atomic}. If not, then $0 < v(\psi_i(a)) \leq \beta$ for $i \in \{1,2\}$, but by Lemma \ref{Lem:Flat}, we have that $\{v(\varphi_{\alpha, \alpha', \varepsilon}(a)) \mid a \in K \}$ is a singleton, a contradiction. Thus, $\varphi_{\alpha, \alpha', \varepsilon}$ is an atom of $\IntR(K,D)$. 
		
		Now we consider $\rho \coloneqq \frac{d}{\varphi_{\alpha, \alpha'}^{\ell-1}}$. Letting $a \in K$, we calculate that
		\[
		v(\rho(a)) = \begin{cases}
			v(d) - (\ell -1)\alpha , & v(a) < \delta,\\
			v(d)-(\ell-1)\alpha' = \alpha, & v(a) > \delta',
		\end{cases}
		\]
		and $\alpha \leq v(\rho(a)) \leq v(d) - (\ell - 1)\alpha$ if $\delta \leq v(a) \leq \delta'$. We also calculate that $v(d) - (\ell-1)\alpha \geq v(d) - (N-3)\alpha \geq 2\alpha$. This shows that $\rho \in \IntR(K,D)$ by Lemma \ref{Lem:BiggestFilter}. We can make a similar argument as the one for $\varphi_{\alpha, \alpha', \varepsilon}$ to show that $\rho$ is an atom. 
		
		Since we have provided a factorization of $d$ in $\IntR(K,D)$ of length $\ell$ for each $\ell \in \{2, \dots, N-2\}$, we have that $\{2, \dots, N-2\} \subseteq \mathcal{L}_{\IntR(K,D)}(d)$. 
		
		\item Now assume that $\Gamma$ is not divisible. Fix $\ell \in \{2, \dots, N-2\}$. Let $\alpha' = \frac{v(d)-\alpha}{\ell-1}$ and $\varepsilon = \frac{\gamma_1}{\ell-1}$. By Lemma \ref{Lem:ZigZag}, there exists $\varphi_{\alpha, \alpha', \varepsilon} \in K(x)$ such that there exist $\delta, \delta' \in \Q\Gamma$ and some $\alpha'' \in \Gamma$ with $\alpha' - \varepsilon < \alpha'' \leq \alpha'$ so that for all $d \in K$, we have
		\[
		v(\varphi_{\alpha, \alpha', \varepsilon}(d)) = \begin{cases}
			\alpha, & \text{if } v(d) < \delta, \\
			\alpha'', & \text{if } v(d) > \delta',
		\end{cases}
		\]
		and $\alpha \leq  v(\varphi_{\alpha, \alpha', \varepsilon}(d)) \leq \alpha''$ if $\delta \leq v(d) \leq \delta'$. Now write again \[\varphi = \underbrace{\varphi_{\alpha, \alpha', \varepsilon} \cdots \varphi_{\alpha, \alpha', \varepsilon}}_{\text{$\ell - 1$ times}} \cdot \frac{d}{\varphi_{\alpha, \alpha', \varepsilon}^{\ell-1}}.\] We can show that $\varphi_{\alpha, \alpha', \varepsilon}$ is an atom of $\IntR(K,D)$ as before. 
		
		Set $\rho = \frac{d}{\varphi_{\alpha, \alpha', \varepsilon}^{\ell-1}}$. Letting $a \in K$, we calculate that
		\[
		v(\rho(a)) = \begin{cases}
			v(d) - (\ell -1)\alpha , & v(a) < \delta,\\
			v(d)-(\ell-1)\alpha'' \geq \alpha, & v(a) > \delta',
		\end{cases}
		\]
		and $\alpha \leq v(\rho(a)) \leq v(d) - (\ell - 1)\alpha$ if $\delta \leq v(a) \leq \delta'$. Since $\alpha' - \varepsilon < \alpha'' \leq \alpha'$, we have that $\alpha \leq v(d) - (\ell-1)\alpha'' < \alpha+\gamma_1$. Utilizing Lemma \ref{Lem:BiggestFilter} shows that $\rho \in \IntR(K,D)$.
		
		Now we write $\rho = \psi_1\psi_2$ for some $\psi_1, \psi_2 \in \IntR(K,D)$. For all $a \in K$ with $v(a)> \delta'$, we have that $\alpha \leq v(\psi_1(a)) + v(\psi_2(a)) \leq \alpha + \gamma_1$. If for such an $a$ neither $v(\psi_1(a))$ nor $v(\psi_2(a))$ is $0$, then we must have $0 < v(\psi_1(a)) \leq \beta$ and $0 < v(\psi_2(a)) \leq \beta$. Otherwise, $v(\psi_i(a)) \geq \alpha$ for some $i \in \{1,2\}$ but then $0 < v(\psi_j(a)) < \gamma_1$ for $j \in \{1,2\}$ such that $j \neq i$, a contradiction. However, now Lemma \ref{Lem:Flat} implies that $\{v(\rho(a)) \mid a \in K\}$ is a singleton, a contradiction. Thus, for some $i \in \{1,2\}$, we have $v(\psi_i(a)) = 0$ for $a \in K$ such that $v(a)$ is sufficiently large. Therefore $\psi_1$ or $\psi_2$ is a unit of $\IntR(K,D)$ by Theorem \ref{Thm:Atomic}.
	\end{enumerate}

	Lastly, if we write $d = \psi_1 \cdots \psi_n$ for atoms $\psi_1, \dots, \psi_n \in \IntR(K,D)$, then we see that $d = \psi_1(0) \cdots \psi_n(0)$ is a product of $n$ nonunit elements in $D$. Therefore, $n \leq L_D(d)$. 
\end{proof}

\section{Positive monoids with one gap}\label{Sect:1toInfinity}

\indent\indent Let $\Gamma$ be an additive subgroup of $\R$ with no minimal strictly positive element. We can scale the elements so that without loss of generality, $1 \in \Gamma$. In this section, we want to focus on integer-valued rational functions over the ring $k[t;M]_{(t;M)}$, where $k$ is a field and $M$ is of the form $M = \{0\} \cup \{\gamma \in \Gamma \mid \gamma \geq 1\}$. The reason for this is that we can always utilize Lemma \ref{Lem:BiggestFilter} to understand the factorization in the base ring $k[t;M]_{(t;M)}$.

\begin{proposition}\label{Prop:FactoringOneInfinity}
	Suppose that $\Gamma$ is an additive subgroup of $\R$ with no minimal strictly positive element and $1 \in \Gamma$. Let  $M = \{0\} \cup \{\gamma \in \Gamma \mid \gamma \geq 1\}$. Take $D$ to be $D = k[t;M]_{(t;M)}$ for some field $k$. Suppose that $v$ is the $M$-valuation on $D$. Then for any nonzero, nonunit element $d \in D$, we have
	\[
	\mathcal{L}_D(d) = \left\{ \lim_{r \to 2^-} \left\lceil \frac{v(d)}{r} \right\rceil, \dots, \lfloor v(d) \rfloor - 1, \lfloor v(d) \rfloor  \right\}
	\]
	and
	\[
	c_D(d) = \begin{cases}
		0,& \text{if $1 \leq v(d) < 2$},\\
		2,& \text{if $2 \leq v(d) < 3$},\\
		3,& \text{if $v(d) \geq 3$}.
	\end{cases}
	\]
\end{proposition}

\begin{proof}
	First note that the atoms of $D$ are exactly the elements of $d \in D$ such that $1 \leq v(d) < 2$. If $v(d) \geq 2$, then $t$ strictly divides $d$ by Lemma \ref{Lem:BiggestFilter}.
	
	Let $d \in D$ be a nonzero, nonunit element. Next, we fix a value for $\ell$ that is in $\left\{ \lim\limits_{r \to 2^-} \left\lceil \frac{v(d)}{r} \right\rceil, \dots, \lfloor v(d) \rfloor - 1, \lfloor v(d) \rfloor  \right\}$. We claim that there is an element $s \in D$ such that $\frac{v(d)-2}{\ell - 1} < v(s) \leq \frac{v(d)-1}{\ell - 1}$ and $1 \leq v(s) < 2$. We have $\ell \leq \lfloor v(d) \rfloor \leq v(d)$. Thus, $1 \leq \frac{v(d)-1}{\ell - 1}$. On the other hand, $ \lim\limits_{r \to 2^-} \left\lceil \frac{v(d)}{r} \right\rceil \leq \ell$. If $v(d) \notin 2\Z$, then $\ell \geq  \left\lceil \frac{v(d)}{2} \right\rceil >  \frac{v(d)}{2}$. If $v(d) \in 2\Z$, then $\lim\limits_{r \to 2^-} \left\lceil \frac{v(d)}{r} \right\rceil = \frac{v(d)}{2} + 1$, so $\ell > \frac{v(d)}{2}$. In either case, we obtain $\ell > \frac{v(d)}{2}$, implying $\frac{v(d)-2}{\ell-1} < 2$. This shows that the intervals $(\frac{v(d)-2}{\ell-1}, \frac{v(d)-1}{\ell-1}]$ and $[1,2)$ have a nonempty intersection. Since $\Gamma$ is a dense subgroup of $\R$, there exists $\gamma \in M$ such that $\gamma \in (\frac{v(d)-2}{\ell-1}, \frac{v(d)-1}{\ell-1}] \cap [1,2)$. Thus, we may take $s = t^\gamma$. Now we can write 
	\[
		d = \underbrace{s \cdots s}_{\text{$\ell -1$ times}} \cdot \frac{d}{s^{\ell - 1}}.
	\]
	Since $1 \leq v(s) < 2$, we know that $s$ is an atom of $S$. Furthermore, $\frac{v(d)-2}{\ell - 1} < v(s) \leq \frac{v(d)-1}{\ell - 1}$ implies that $1 \leq v(d) - (\ell - 1)v(s) < 2$, so $\frac{d}{v(s)^{\ell-1}}$ is also an atom of $S$. 
	
	Now we want to show that $d$ has no other factorization lengths. Suppose that $d = a_1\cdots a_\ell$ for some $\ell \in \N$ and each $a_i \in D$ is irreducible. Then $v(a_i) \in [1,2)$ for each $i$ so $v(d) \in [\ell, 2\ell)$. Since $\ell\leq v(d)$, we must have $\ell \leq \lfloor v(d) \rfloor$. On the other hand, we have $\frac{v(d)}{2} < \ell$. Thus, $\lim\limits_{r \to 2^-} \left\lceil \frac{v(d)}{r} \right\rceil \leq \ell$. This shows that 
	\[
	\mathcal{L}_D(d) = \left\{ \lim_{r \to 2^-} \left\lceil \frac{v(d)}{r} \right\rceil, \dots, \lfloor v(d) \rfloor - 1, \lfloor v(d) \rfloor  \right\}
	\]

	Now we calculate the catenary degree. If $1 \leq v(d) < 2$, then $d$ is an atom of $D$ so the catenary degree is $0$. If $2 \leq v(d) < 3$, then $\mathcal{L}_D(d) = \{2\}$. Moreover, we have that for any distinct $\alpha, \beta \in M$ such that $1 < \alpha, \beta < 2$, the elements $t + t^\alpha$ and $t + t^\beta$ are not associates. Otherwise, $\frac{t+t^\alpha}{t+t^\beta}-1 = \frac{t^{\alpha-1}+t^{\beta-1}}{1+t^{\beta-1}} \in D$, but the valuation of this element is $\min\{\alpha-1,\beta-1\}$ which is not in $M$, a contradiction. Thus, for $d \in D$ such that $2 \leq v(d) < 3$, since $t + t^{\alpha}$ for any $\alpha \in M$ such that $1 < \alpha < 2$ is a distinct divisor of $d$ up to association, we know that $\abs{\mathsf{Z}_D(d)} > 1$, meaning that $c_D(d) = 2$. 
	
	Now we assume that $d \in D\setminus\{0\}$ is such that $v(d) \geq 3$. We will establish a $3$-chain from any factorization of $d$ to the factorization of $d$ given by \[d = \underbrace{t\cdots t}_{\text{$\lfloor v(d) \rfloor-1$ times}} \frac{d}{t^{\lfloor v(d) \rfloor-1}}.\]
	Once we establish these 3-chains, since $\abs{\mathcal{L}_D(d)} > 1$, we know that $c_D(d) = 3$ since a 2-chain can only be established between factorizations of the same length. Take a factorization $d = a_1 \cdots a_n$, where each $a_i \in D$ is irreducible. If there does not exist distinct $i, j \in \{1, \dots, n\}$ such that $a_i, a_j \not\sim t$, then the factorization $a_1 \cdots a_n$ is the same as $t \cdots t \cdot  \frac{d}{t^{\lfloor v(d) \rfloor-1}}$ up to reordering and association. Now suppose that there exist distinct $i, j \in \{1, \dots, n\}$ such that $a_i, a_j \not\sim t$. Note that $2 \leq v(a_ia_j) < 4$. We know that
	\[
	a_i a_j = \begin{cases}
		t \cdot t \cdot \frac{a_i a_j}{t^2}, & \text{if $3 \leq v(a_ia_j)< 4$},\\
		t \cdot \frac{a_ia_j}{t}, & \text{if $2 \leq v(a_ia_j) < 3$}. 
	\end{cases}
	\]
	In either case, there is a factorization of $d$ that is of distance at most $3$ from $a_1 \cdots a_n$ and has strictly more factors associate to $t$. Thus, there is a $3$-chain between any factorization of $d$. As established before, this implies that $c_D(d) = 3$.  
\end{proof}

Let $D$ be as in the previous proposition and let $K$ be its field of fractions. We compare the factorization of elements $d \in D$ viewed as elements of $D$ and viewed as elements of $\IntR(K,D)$. We see that the set of lengths extends downward to 2, but the catenary degree remains the same. 

\begin{theorem}
	Suppose that $\Gamma$ is an additive subgroup of $\R$ with no minimal strictly positive element and $1 \in \Gamma$. Let  $M = \{0\} \cup \{\gamma \in \Gamma \mid \gamma \geq 1\}$. Take $D$ to be $D = k[t;M]_{(t;M)}$ for some field $k$. Denote by $K$ the field of fractions of $D$. Suppose that $v$ is the $M$-valuation on $K$. Also suppose that $k$ is not algebraically closed or $\mathsf{gp}(M)$ is not divisible. Then for any nonzero, nonunit element $d \in D$, we have
	\[
	\mathcal{L}_{\IntR(K,D)}(d) = \begin{cases}
		\left\{ 2, 3, \dots, \lfloor v(d) \rfloor - 1, \lfloor v(d) \rfloor  \right\}, & \text{if $v(d) \geq 2$},\\
		\left\{1 \right\}, & \text{if $1 \leq v(d) < 2$}
	\end{cases}
	\]
	and
	\[
	c_{\IntR(K,D)}(d) = \begin{cases}
		0,& \text{if $1 \leq v(d) < 2$},\\
		2,& \text{if $2 \leq v(d) < 3$},\\
		3,& \text{if $v(d) \geq 3$}.
	\end{cases}
	\]
	Moreover, for a nonzero, nonunit element $\varphi \in \IntR(K,D)$, we have \[\mathcal{L}_{\IntR(K,D)}(\varphi) =\begin{cases}
		\left\{2,3, \dots, \lfloor \alpha(\varphi)\rfloor \right\}, & \text{if $\alpha(\varphi) \geq 2$}\\
		\{1\}, & \text{if $1 \leq \alpha(\varphi) <2$},
	\end{cases}\] where $\alpha(\varphi) = \inf\{v(\varphi(a)) \mid a \in K \}$. 
\end{theorem}

\begin{proof}
	Let $d \in D$ be a nonzero, nonunit element. Suppose $1 \leq v(d) < 2$. Then Lemma \ref{Lem:AtomTransfer} implies that $d$ is irreducible in $\IntR(K,D)$. Suppose that $2 \leq v(d) \leq 3$. Then we have $\{2, \dots, \left\lfloor v(d) \right\rfloor \} = \mathcal{L}_{D}(d) \subseteq \mathcal{L}_{\IntR(K,D)}(d)$ by Proposition \ref{Prop:FactoringOneInfinity}. Furthermore, if $d = \psi_1 \cdots \psi_n$ for some irreducible elements $\psi_1, \dots, \psi_n \in \IntR(K,D)$, we have $d = \psi_1(0) \cdots \psi_n(0)$, so we have $n \leq L_D(d) =  \left\lfloor v(d) \right\rfloor$. Therefore, \[\mathcal{L}_{\IntR(K,D)}(d) = \left\{ 2, 3, \dots, \lfloor v(d) \rfloor - 1, \lfloor v(d) \rfloor  \right\},\]
	
	If $d \in D$ is nonzero and has the property that $v(d) > 3$, then we use Theorem \ref{Thm:ExtendSetOfLengths}. We see that $\min\{n \in \N \mid n\cdot 1 > v(d) \} \geq \lceil v(d) \rceil$. Thus, we have
	\[
	\mathcal{L}_D(d) \cup \{2, \dots, \lceil v(d) \rceil - 2 \} \subseteq \mathcal{L}_{\IntR(K,D)}(d) \subseteq \{2, \dots, \left\lfloor v(d) \right\rfloor  \}
	\]
	Since Proposition \ref{Prop:FactoringOneInfinity} gives $\mathcal{L}_D(d)$, we have that \[	\mathcal{L}_{\IntR(K,D)}(d) = \left\{ 2, 3, \dots, \lfloor v(d) \rfloor - 1, \lfloor v(d) \rfloor  \right\}.\]
	
	Now we note that for any nonzero, nonunit element $\varphi \in \IntR(K,D)$, the rational function $\varphi$ is an atom if and only if $1 \leq \alpha(\varphi) < 2$. If $\varphi \in \IntR(K,D)$ is such that $1 \leq \alpha(\varphi) < 2$, then there exists $a \in K$ such that $1 \leq v(\varphi(a)) < 2$. Thus, whenever we write $\varphi = \psi_1\psi_2$ for some $\psi_1, \psi_2 \in \IntR(K,D)$, we have that $\varphi(a) = \psi_1(a)\psi_2(a)$. Since $\varphi(a)$ is an atom of $D$, we must have $\psi_1(a)$ or $\psi_2(a)$ be a unit of $D$, which means that $\psi_1$ or $\psi_2$ is a unit of $\IntR(K,D)$ by Theorem \ref{Thm:Atomic}. On the other hand, suppose that $\varphi \in \IntR(K,D)$ is an atom. Then $\varphi(0)$ is an atom of $D$, which means that $1 \leq v(\varphi(0)) < 2$. Thus, $1 \leq \alpha(\varphi) < 2$. 
	
	Suppose that $\varphi \in \IntR(K,D)$ is not an atom of $\IntR(K,D)$. Then $\lfloor \alpha(\varphi) \rfloor \geq 2$. We can write $t^{\lfloor \alpha(\varphi)\rfloor}$ as the product of $\ell$ atoms of $\IntR(K,D)$, where $\ell \in \{2, 3, \dots, \lfloor \alpha(\varphi)\rfloor \}$. Fix such an $\ell$. We have $t^{\lfloor \alpha(\varphi)\rfloor} = \psi_1\cdots \psi_\ell$, where $\psi_1, \dots, \psi_\ell$ are atoms of $\IntR(K,D)$. We also know that there exists $a \in K$ such that $\lfloor \alpha(\varphi) \rfloor \leq v(\varphi(a)) < \lfloor \alpha(\varphi) \rfloor + 1$. From the proof of Theorem \ref{Thm:ExtendSetOfLengths}, we can write the factorization in a way so that $v(\psi_1(a)) =1$. Then $1 \leq \alpha\left(\frac{\varphi}{\psi_2\cdots\psi_\ell}\right) \leq v(\varphi(a)) - (\lfloor \alpha(\varphi) \rfloor - 1) < 2$, meaning that $\frac{\varphi}{\psi_2\cdots\psi_\ell}$ is an atom of $\IntR(K,D)$. Therefore, we can write $\varphi = \frac{\varphi}{\psi_2\cdots\psi_\ell} \cdot \psi_2 \cdots \psi_\ell$. Now, we obtain $\{2, 3, \dots, \lfloor \alpha(\varphi)\rfloor \} \subseteq \mathcal{L}_{\IntR(K,D)}(d)$. 
	
	Write $\varphi = \psi_1 \cdots \psi_n$ for atom $\psi_1, \dots, \psi_n \in \IntR(K,D)$. There exists $a \in K$ such that $\lfloor \alpha(\varphi) \rfloor \leq v(\varphi(a)) < \lfloor \alpha(\varphi) \rfloor + 1$. Thus, $v(\varphi(a)) = v(\psi_1(a)) + \cdots + v(\psi_n(a))$ implies that $n \leq \lfloor \alpha(\varphi) \rfloor$ since $v(\psi_i(a)) \geq 1$ for each $i$. We can then conclude that $\mathcal{L}_{\IntR(K,D)}(\varphi) = \left\{2,3, \dots, \lfloor \alpha(\varphi)\rfloor \right\}$.
	
	Lastly, we calculate catenary degrees. Let $d \in D$ be a nonzero, nonunit element. As we have seen before, if $1 \leq v(d) < 2$, then $d$ is an atom of $\IntR(K,D)$ and therefore $c_{\IntR(K,D)}(d) = 0$. If $2 \leq v(d) < 3$, then $\mathcal{L}_{\IntR(K,D)}(d) = \{2\}$. Moreover, we can show that $t + t^\gamma$ and $t + t^{\gamma'}$ for distinct $\gamma, \gamma' \in M$ such that $1 < \gamma, \gamma' < 2$ are not associates in $\IntR(K,D)$. As a consequence, we have that $\abs{\mathsf{Z}_{\IntR(K,D)}(d)} > 1$, so $c_{\IntR(K,D)}(d) = 2$.
	
	Now we assume that $d \in D$ is a nonzero element such that $v(d) \geq 3$. We want to show that there exists a $3$-chain from any factorization of $d$ in $\IntR(K,D)$ to the factorization of $d$ given by
	\[d = \underbrace{t\cdots t}_{\text{$\lfloor v(d) \rfloor-1$ times}} \frac{d}{t^{\lfloor v(d) \rfloor-1}}.\]
	Write $d = \psi_1 \cdots \psi_n$ for $\psi_1, \dots, \psi_n$ irreducible elements of $\IntR(K,D)$. If there do not exist distinct $i, j$ in $\{1, \dots, n\}$ such that $\psi_i$ and $\psi_j$ are both not associate to $t$, then the factorization $d = \psi_1 \cdots \psi_n$ is the same as $t \cdots t \cdot \frac{d}{t^{\lfloor v(d) \rfloor-1}}$ up to reordering and association. Now suppose that there are distinct $i, j$ in $\{1, \dots, n\}$ such that $\psi_i$ and $\psi_j$ are both not associate to $t$. Note that $\alpha(\psi_i\psi_j)  \geq 2$. Thus, $\frac{\psi_i\psi_j}{t} \in \IntR(K,D)$. We have that $\alpha\left(\frac{\psi_i\psi_j}{t} \right) = \alpha(\psi_i\psi_j) - 1 \geq 1$. If $2 \leq \alpha(\psi_i\psi_j) <3 $, then $\frac{\psi_i\psi_j}{t}$ is irreducible. If $\alpha(\psi_i\psi_j) \geq 3$, then $2 \in \mathcal{L}_{\IntR(K,D)}\left(\frac{\psi_i\psi_j}{t}\right)$, so we can write $\frac{\psi_i\psi_j}{t} = \rho_1\rho_2$ for some irreducible elements $\rho_1, \rho_2 \in \IntR(K,D)$.  Now, we observe that
	\[
	\psi_i\psi_j = \begin{cases}
		t \cdot \frac{\psi_i\psi_j}{t}, & \text{if $2 \leq \alpha(\psi_i\psi_j) <3 $},\\
		t \cdot \rho_1 \cdot \rho_2, & \text{if $\alpha(\psi_i\psi_j) \geq 3$}.
	\end{cases}
	\]
	This means that there is a factorization of distance at most $3$ from $d = \psi_1 \cdots \psi_n$ such that there are strictly more factors associate to $t$. This shows that there is a $3$-chain between any two factorizations of $d$ in $\IntR(K,D)$, showing that $c_{\IntR(K,D)}(d) \leq 3$. Since $\abs{\mathcal{L}_{\IntR(K,D)}(d)} > 1$, we must have $c_{\IntR(K,D)}(d) = 3$ as the distance between two factorizations of different lengths is at least $3$. 
\end{proof}

\bibliographystyle{amsalpha}
\bibliography{references}
\end{document}